\documentclass[a4paper]{article}%
\usepackage{amssymb,latexsym}
\usepackage{amsfonts}
\usepackage{amsmath}
\usepackage[colorlinks=true, pdfstartview=FitV, linkcolor=blue,citecolor=blue,
urlcolor=blue]{hyperref}
\usepackage{color}
\usepackage{amssymb}
\usepackage{graphicx}%
\providecommand{\U}[1]{\protect\rule{.1in}{.1in}}
\newtheorem{theorem}{Theorem}

\newtheorem{definition}[theorem]{Definition}
\newtheorem{example}[theorem]{Example}

\newtheorem{lemma}[theorem]{Lemma}

\newtheorem{remark}[theorem]{Remark}

\newenvironment{proof}[1][Proof]{\noindent\textbf{#1.} }{\ \rule{0.5em}{0.5em}}
\begin{document}

\title{Periodic analogues of Dedekind sums and transformation formulas of Eisenstein series}
\author{M. Cihat Da\u{g}l\i\thanks{Corresponding author}\ \ and M\"{u}m\"{u}n Can\\Department of Mathematics, Akdeniz University, \\07058-Antalya, Turkey\\e-mails: mcihatdagli@akdeniz.edu.tr, mcan@akdeniz.edu.tr}
\date{}
\maketitle

\begin{abstract}
In this paper, a transformation formula under modular substitutions is derived
for a large class of generalized Eisenstein series. Appearing in the
transformation formulae are generalizations of Dedekind sums involving the
periodic Bernoulli function. Reciprocity theorems are proved for these
Dedekind sums. Furthermore, as an application of the transformation formulae,
relations between various infinite series and evaluations of several infinite
series are deduced. Finally, we consider these sums for some special cases.



\end{abstract}

\section{Introduction}

For integers $c$ and $d$ with $c>0$, the classical Dedekind sum $s(d,c)$ is
defined by%
\[
s(d,c)=\sum\limits_{n(\operatorname{mod}c)}\left(  \left(  \frac{n}{c}\right)
\right)  \left(  \left(  \frac{dn}{c}\right)  \right)
\]
where the \textit{sawtooth} function is defined by%
\[
\left(  \left(  x\right)  \right)  =%
\begin{cases}
x-\left[  x\right]  -1/2, & \text{if }x\in\mathbb{R}\backslash\mathbb{Z}%
\text{;}\\
0, & \text{if }x\in\mathbb{Z}\text{,}%
\end{cases}
\]
with $[x]$\ the floor function. One of the most important properties of
Dedekind sums is the reciprocity formula\textbf{\ }%
\begin{equation}
s(d,c)+s(c,d)=-\frac{1}{4}+\frac{1}{12}\left(  \frac{d}{c}+\frac{c}{d}%
+\frac{1}{dc}\right) \label{35}%
\end{equation}
whenever $c$ and $d$ are coprime positive integers. For several proofs of
(\ref{35}) and generalizations for Dedekind sums, for example, see
\cite{14,12,2,11,7,15,20,5,18,ham,lim,9,13,sek,16}.

Dedekind sums first arose in the transformation formulas of $\log\eta(z),$
where $\eta(z)$ denotes the Dedekind eta-function. There are several other
functions such as Eisenstein series, which possess transformation formula
similar to $\log\eta(z)$. Lewittes \cite{6} has discovered a method of
obtaining transformation formulas for certain generalized Eisenstein series.

In \cite{12}, Berndt gave a different account of the final part of Lewittes'
proof. His new proof yielded elegant transformation formulas in which Dedekind
sums or various generalizations of Dedekind sums appear. The results of
\cite{12} have been generalized in \cite{11}. Berndt \cite{22} derived a
number of transformation formulas from the general theorem in \cite{11}.
Arising in the transformation formulae are various types of Dedekind sums, all
of which satisfy reciprocity theorems.

In \cite{2} and \cite{7}, Berndt considered a more general class of Eisenstein
series and developed transformation formulae for a wide class of functions
involving characters including the natural character generalizations of
$\log\eta(z)$. In these formulas further generalizations of Dedekind sums
appear. These sums involve characters and generalized Bernoulli functions, and
possess reciprocity laws proved via transformation formulae.

Also in \cite{8} and \cite{22}, Berndt has used the transformation formulas to
evaluate several classes of infinite series and establish many relations
between various infinite series.

The aim of this paper is to obtain a transformation formula for a very large
class of Eisenstein series defined by
\begin{equation}
G(z,s;A_{\alpha},B_{\beta};r_{1},r_{2})=\sum\limits_{m,n=-\infty}^{\infty
}\ \hspace{-0.19in}^{^{\prime}}\frac{f(\alpha m)f^{\ast}(\beta n)}%
{((m+r_{1})z+n+r_{2})^{s}},\text{ }\operatorname{Re}(s)>2,\text{
}\operatorname{Im}(z)>0\label{53}%
\end{equation}
where $\left\{  f(n)\right\}  $ and $\left\{  f^{\ast}(n)\right\}  ,$
$-\infty<n<\infty$ are sequences of complex numbers with period $k>0,$ and
$A_{\alpha}=\left\{  f(\alpha n)\right\}  $ and $B_{\beta}=\left\{  f^{\ast
}(\beta n)\right\}  .$ In (\ref{53}), the dash $\prime$ means that the
possible pair $m=-r_{1},n=-r_{2}$ is excluded from the summation.
Generalizations of Dedekind sums involving the periodic Bernoulli function
appear in the transformation formulae. It is shown that these Dedekind sums
obey reciprocity theorems. Moreover, transformation formulas contain many
other interesting results as special cases. These results give the values of
several interesting infinite series and yield relations between various
infinite series, some of which generalize some results of \cite{8, 22} and
also found in Ramanujan's Notebooks; such as%
\begin{equation}
\sum\limits_{n=1}^{\infty}\frac{\chi\left(  n\right)  }{n\left(  e^{n\gamma
/2}+e^{-n\gamma/2}\right)  }+\sum\limits_{n=1}^{\infty}\frac{\chi\left(
n\right)  }{n\left(  e^{n\theta/2}+e^{-n\theta/2}\right)  }=\frac{\pi}%
{8},\nonumber
\end{equation}
which gives Entry 25 (vii) on p. 295 of Berndt \cite{17} for $\theta
=\gamma=\pi$ (see Example \ref{ex1}). Finally, we consider these Dedekind sums
for some special values of $A_{\alpha}$ and $B_{\beta}.$

We summarize this study as follows: Section 2 is a preliminary section
containing the basic definitions and notations we need. In Section 3, we show
that the function $G(z,s;A_{\alpha},B_{\beta};r_{1},r_{2})$ can be
analytically continued to the entire $s$--plane with the possible exception of
a simple pole at $s=1$. Section 4 deals with a transformation formula for the
function $G(z,s;A,B;r_{1},r_{2})$. In Section 5, we first investigate the
transformation formula for the case $s=r_{1}=r_{2}=0,$ in which a
generalization of Dedekind sum, called periodic Dedekind sums, appears.
Setting $s=0$ and $r_{1},$ $r_{2}$ arbitrary real numbers in the
transformation formula, new generalization of Dedekind sum arises, as well.
Moreover, we prove the reciprocity theorems for these Dedekind sums. In
Section 6, concerning with some special cases of transformation formulae
allows us to present several relations between various infinite series. In
Section 7, the periodic Dedekind sums is illustrated for some special values
of $A_{\alpha}$ and $B_{\beta}.$

The definition of the Eisenstein series in (\ref{53}) and the methods
presented in the sequel are motivated by \cite{2} and \cite{7}.

\section{Preliminaries}

In this section, we give a brief summary for the material only needed in the
subsequent sections. Throughout this study we use the modular transformation
$Vz=V\left(  z\right)  =\left(  az+b\right)  /\left(  cz+d\right)  $ where
$a,$ $b, $ $c$ and $d$ are integers with $ad-bc=1$ and $c>0$. We use the
notation $\left\{  x\right\}  $ for the fractional part of $x,$ and
$\lambda_{x}$ for the characteristic function of integers. The upper
half-plane $\left\{  x+iy:y>0\right\}  $ will be denoted by $\mathbb{H}$ and
the upper quarter--plane $\{x+iy:x>-d/c$, $y>0\}$ by $\mathbb{K}$. We put
$e(z)=e^{2\pi iz}$ and unless otherwise stated, we use the branch of the
argument defined by $-\pi\leq$ arg $z<\pi$.

Let $\left\{  f(n)\right\}  =A,$ $-\infty<n<\infty$ be sequence of complex
numbers with period $k>0.$ For $\left\vert t\right\vert <2\pi/k,$ the periodic
Bernoulli numbers and polynomials are defined by means of the generating
functions \cite{10}%
\begin{equation}
\sum\limits_{n=0}^{k-1}\frac{tf(n)e^{nt}}{e^{kt}-1}=\sum\limits_{j=0}^{\infty
}\frac{B_{j}(A)}{j!}t^{j}\label{b}%
\end{equation}
and
\begin{equation}
\sum\limits_{n=0}^{k-1}\frac{tf(-n)e^{(n+x)t}}{e^{kt}-1}=\sum\limits_{j=0}%
^{\infty}\frac{B_{j}(x,A)}{j!}t^{j}.\label{bp}%
\end{equation}
Note that, when $A=I=\left\{  1\right\}  ,$(\ref{b}) and (\ref{bp}) reduce to
ordinary Bernoulli numbers and polynomials, defined by the generating
functions \cite{1}%
\begin{align}
\frac{t}{e^{t}-1}  &  =\sum\limits_{n=0}^{\infty}B_{n}\frac{t^{n}}%
{n!},\ |t|<2\pi,\nonumber\\
\frac{te^{xt}}{e^{t}-1}  &  =\sum\limits_{n=0}^{\infty}B_{n}(x)\frac{t^{n}%
}{n!},\ |t|<2\pi,\label{8}%
\end{align}
respectively. Notice that\textbf{\ }$B_{0}(x)=1,$ $B_{1}=-1/2,$ $B_{1}%
(1)=1/2,$ $B_{2n+1}=B_{2n-1}\left(  1/2\right)  =0,$\ $n\geq1,$\ and
$B_{1}(x)=x-1/2.$

Throughout this paper, the $n-$th Bernoulli function will be denoted by
$P_{n}\left(  x\right)  $ and is defined by
\[
n!P_{n}\left(  x\right)  =B_{n}\left(  x-\left[  x\right]  \right)  .
\]
In particular $P_{1}\left(  x\right)  =x-\left[  x\right]  -1/2.$ These
functions satisfy\textbf{\ }the Raabe or multiplication formula for all real
$x$
\begin{equation}
P_{n}(x)=r^{n-1}\sum\limits_{m=0}^{r-1}P_{n}\left(  \frac{m+x}{r}\right)
\label{26}%
\end{equation}
and the reflection identity $P_{n}(-x)=\left(  -1\right)  ^{n}P_{n}(x)$ except
$n=1$ and $x\in\mathbb{Z}$, in that case
\[
P_{1}(-x)=P_{1}(x)=P_{1}(0)=-1/2.
\]

The periodic Bernoulli functions $P_{n}(x,A),$ are functions with period $k$,
may be defined by \cite{10}
\begin{equation}
P_{0}(x,A)=B_{0}(A)=\frac{1}{k}\sum\limits_{m=0}^{k-1}f(m)\label{10}%
\end{equation}
and%
\begin{equation}
P_{n}(x,A)=k^{n-1}\sum\limits_{m=0}^{k-1}f(-m)P_{n}\left(  \frac{m+x}%
{k}\right)  ,\text{ }n\geq1\label{10-1}%
\end{equation}
for all real $x$.

Define the sequence $\widehat{A}=\left\{  \widehat{f}(n)\right\}  $ by
\begin{equation}
\widehat{f}(n)=\frac{1}{k}\sum\limits_{j=0}^{k-1}f\left(  j\right)  e\left(
-nj/k\right) \label{47}%
\end{equation}
for $-n<\infty<n.$ These are the finite Fourier series coefficients of
$\left\{  f(n)\right\}  .$ Clearly $\widehat{A}$ also has period $k.$ Note
that (\ref{47}) holds if and only if%
\begin{equation}
f(n)=\sum\limits_{j=0}^{k-1}\widehat{f}\left(  j\right)  e\left(  nj/k\right)
,\text{ \ }-n<\infty<n.\text{\ }\label{50}%
\end{equation}

\section{The function $G(z,s;A_{\alpha},B_{\beta};r_{1},r_{2})$}

Let $A_{\alpha}$ denote the sequence $\left\{  f\left(  \alpha n\right)
\right\}  ,$ $-\infty<n<\infty,$ i.e., $\left\{  f\left(  \alpha n\right)
\right\}  =A_{\alpha},$ $\alpha\in\mathbb{Z}$.\ Similarly $\left\{  f^{\ast
}\left(  \alpha n\right)  \right\}  =B_{\alpha},$ $-\infty<n<\infty.$ We begin
with a study of the function%
\[
G(z,s;A_{\alpha},B_{\beta};r_{1},r_{2})=\sum\limits_{m,n=-\infty}^{\infty
}\frac{f(\alpha m)f^{\ast}(\beta n)}{((m+r_{1})z+n+r_{2})^{s}}.
\]
From the definition, we see that%
\begin{align}
&  G(z,s;A_{\alpha},B_{\beta};r_{1},r_{2})\nonumber\\
&  =\lambda_{r_{1}}f(-\alpha r_{1})\sum\limits_{n=-\infty}^{\infty}f^{\ast
}(\beta n)(n+r_{2})^{-s}+\left(  \sum\limits_{m<-r_{1}}\sum\limits_{n=-\infty
}^{\infty}+\sum\limits_{m>-r_{1}}\sum\limits_{n=-\infty}^{\infty}\right)
\frac{f(\alpha m)f^{\ast}(\beta n)}{((m+r_{1})z+n+r_{2})^{s}}\nonumber\\
&  =S_{1}+S_{2}+S_{3}.\label{1}%
\end{align}
Firstly, we write $S_{1}$ as
\begin{align}
S_{1} &  =\lambda_{r_{1}}f(-\alpha r_{1})\left(  \sum\limits_{n>-r_{2}}%
f^{\ast}(\beta n)(n+r_{2})^{-s}+\sum\limits_{n>r_{2}}f^{\ast}(-\beta
n)(-n+r_{2})^{-s}\right) \nonumber\\
&  =\lambda_{r_{1}}f(-\alpha r_{1})\left(  L(s,B_{\beta};r_{2}%
)+e(s/2)L(s,B_{-\beta};-r_{2})\right) \label{2}%
\end{align}
where
\begin{equation}
L(s;A_{\beta};\theta)=\sum\limits_{n>-\theta}f(n\beta)(n+\theta)^{-s},\text{
for }\operatorname{Re}\left(  s\right)  >1\text{ and }\theta\text{
real.}\label{12}%
\end{equation}
$L(s;A_{\beta};\theta)$ can be written in terms of Hurwitz zeta function
$\zeta\left(  s,\theta\right)  $ as follows: Setting $n=mk+j+\left[
-\theta\right]  +1,$ $0\leq j\leq k-1,$ $0\leq m<\infty$ and using the fact
$\left[  \theta\right]  +\left[  -\theta\right]  =\lambda_{\theta}-1,$\ it is
seen that for $\operatorname{Re}(s)>1$%

\begin{align}
L(s;A_{\beta};\theta) &  =\sum\limits_{j=0}^{k-1}\sum\limits_{m=0}^{\infty
}\frac{f(\beta\left(  j+\left[  -\theta\right]  +1\right)  )}{\left(
km+j+\left[  -\theta\right]  +1+\theta\right)  ^{s}}\nonumber\\
&  =\frac{1}{k^{s}}\sum\limits_{j=1}^{k}f(\beta\left(  j-\left[
\theta\right]  +\lambda_{\theta}\right)  )\zeta\left(  s,\frac{j+\left\{
\theta\right\}  +\lambda_{\theta}}{k}\right)  .\label{Ls}%
\end{align}
Since the periodic zeta-function $\zeta\left(  s,A\right)  $ has an analytic
continuation into the entire $s$--plane where it is holomorphic with the
possible exception of a simple pole at $s=1$ where the residue is
$\widehat{f}\left(  0\right)  =B_{0}\left(  A\right)  $ (see \cite[Corollary
6.5]{10}), $L(s;A_{\beta};\theta)$ can be analytically continued to the entire
$s$--plane except for $s=1.$

Secondly, if we replace $m$ by $-m$ and $n$ by $-n$ in $S_{2},$ then%
\begin{align*}
S_{2}  &  =\sum\limits_{m<-r_{1}}f(\alpha m)\sum\limits_{n=-\infty}^{\infty
}\frac{f^{\ast}(\beta n)}{((m+r_{1})z+n+r_{2})^{s}}\\
&  =e(s/2)\sum\limits_{m>r_{1}}f(-\alpha m)\sum\limits_{n=-\infty}^{\infty
}\frac{f^{\ast}(-\beta n)}{((m-r_{1})z+n-r_{2})^{s}}.
\end{align*}
Using (\ref{50}) and the Lipschitz summation formula given by
\[
\sum\limits_{n=1}^{\infty}(n-\tau)^{s-1}e^{2\pi iz(n-\tau)}=\frac{\Gamma
(s)}{\left(  -2\pi i\right)  ^{s}}\sum\limits_{n=-\infty}^{\infty}%
(z+n)^{-s}e^{2\pi in\tau}%
\]
where $\operatorname{Re}(s)>1,$ $z\in\mathbb{H}$ and $0\leq\tau<1,$ we obtain
\begin{align*}
&  \Gamma(s)\sum\limits_{n=-\infty}^{\infty}\frac{f^{\ast}(-\beta
n)}{(w+n)^{s}}=\sum\limits_{j=0}^{k-1}\frac{f^{\ast}(-\beta j)}{k^{s}}%
\Gamma(s)\sum\limits_{n=-\infty}^{\infty}\frac{1}{(n+\frac{j+w}{k})^{s}}\\
&  \quad=\left(  -2\pi i/k\right)  ^{s}\sum\limits_{n=1}^{\infty}%
n^{s-1}e\left(  \frac{nw}{k}\right)  \sum\limits_{j=0}^{k-1}f^{\ast}(-\beta
j)e\left(  \frac{n\beta^{-1}\beta j}{k}\right) \\
&  \quad=\left(  -2\pi i/k\right)  ^{s}\sum\limits_{n=1}^{\infty}%
n^{s-1}e\left(  \frac{nw}{k}\right)  \sum\limits_{j=0}^{k-1}f^{\ast
}(j)e\left(  -n\beta^{-1}\frac{j}{k}\right) \\
&  \quad=k\left(  -2\pi i/k\right)  ^{s}\sum\limits_{n=1}^{\infty
}\widehat{f^{\ast}}(\beta^{-1}n)n^{s-1}e\left(  \frac{nw}{k}\right)
\end{align*}
for $\operatorname{Re}(s)>1$ , $\operatorname{Im}(w)>0$ and $\beta\beta
^{-1}\equiv1\left(  \operatorname{mod}k\right)  $. Thus, we have
\begin{equation}
S_{2}=\frac{\left(  -2\pi i/k\right)  ^{s}}{\Gamma(s)}ke(s/2)A\left(
z,s;A_{-\alpha},\widehat{B}_{\beta^{-1}};-r_{1}{\LARGE ,}-r_{2}\right)
\label{3}%
\end{equation}
where%
\[
A(z,s;A_{\alpha},A_{\beta};r_{1}{\huge ,}r_{2})=\sum\limits_{m>-r_{1}}f(\alpha
m)\sum\limits_{n=1}^{\infty}f(\beta n)e\left(  n\frac{(m+r_{1})z+r_{2}}%
{k}\right)  n^{s-1}.
\]

Similarly we deduce that
\begin{equation}
S_{3}=\frac{\left(  -2\pi i/k\right)  ^{s}}{\Gamma(s)}kA\left(  z,s;A_{\alpha
},\widehat{B}_{-\beta^{-1}};r_{1}{\huge ,}r_{2}\right)  .\label{4}%
\end{equation}

Combining (\ref{1}), (\ref{2}), (\ref{3}) and (\ref{4}), we conclude that%
\begin{align}
&  G(z,s;A_{\alpha},B_{\beta};r_{1},r_{2})\nonumber\\
&  =\frac{\left(  -2\pi i/k\right)  ^{s}k}{\Gamma(s)}\left(  A\left(
z,s;A_{\alpha},\widehat{B}_{-\beta^{-1}};r_{1}{\huge ,}r_{2}\right)
+e(s/2)A\left(  z,s;A_{-\alpha},\widehat{B}_{\beta^{-1}};-r_{1}{\LARGE ,}%
-r_{2}\right)  \right) \nonumber\\
&  \quad+\lambda_{r_{1}}f(-\alpha r_{1})\left(  L(s;B_{\beta};r_{2}%
)+e(s/2)L(s;B_{-\beta};-r_{2})\right)  .\label{11}%
\end{align}
Since $L(s;B_{\beta};r_{2})$ can be analytically continued to the entire
$s$--plane with the possible exception $s=1$ and since $A\left(
z,s;A_{\alpha},B_{\beta};r_{1}{\huge ,}r_{2}\right)  $ is entire function of
$s,$ $G(z,s;A_{\alpha},B_{\beta};r_{1},r_{2})$ can be analytically continued
to the entire $s$--plane with the possible exception $s=1$.

For simplicity, the function $G(z,s;A_{\alpha},B_{\beta};0,0)$\ will be
denoted by $G(z,s;A_{\alpha},B_{\beta})$ and $G(z,s;A_{1},$ $B_{1};r_{1}%
,r_{2})=G\left(  z,s;A,B;r_{1},r_{2}\right)  $.

\section{Transformation Formulas}

In this section, we present transformation formulas for the function $G\left(
z,s;A_{\alpha},B_{\beta};r_{1},r_{2}\right)  .$ We need the following lemma.

\begin{lemma}
(\cite{6})\label{le1} Let $E,$ $F,$ $C$ and $D$ be real with $E$ and $F$ not
both zero and $C>0$. Then for $z\in\mathbb{H},$%
\[
arg\left(  \left(  Ez+F\right)  /\left(  Cz+D\right)  \right)  =arg\left(
Ez+F\right)  -arg\left(  Cz+D\right)  +2\pi l,
\]
where $l$ is independent of $z$ and $l=\left\{
\begin{array}
[c]{ll}%
1, & E\leq0\text{ and }DE-CF>0,\\
0, & \text{otherwise.}%
\end{array}
\right.  $
\end{lemma}

\begin{theorem}
\label{mainteo} Define $R_{1}=ar_{1}+cr_{2}${\huge \ }and $R_{2}=br_{1}%
+dr_{2},$ in which $r_{1}$ and $r_{2}$ are arbitrary real numbers. Let
$\rho=\rho\left(  R_{1},R_{2},c,d\right)  =\left\{  R_{2}\right\}  c-\left\{
R_{1}\right\}  d.$ Suppose first that $a\equiv d\equiv0\left(
\operatorname{mod}k\right)  .$ Then for $z\in\mathbb{K}$ and all $s,$
\begin{align}
&  (cz+d)^{-s}\Gamma(s)G(Vz,s;A,B;r_{1},r_{2})\label{13-a}\\
&  =\Gamma(s)G(z,s;B_{-b},A_{-c};R_{1},R_{2})-2i\Gamma(s)\sin(\pi
s)L(s;A_{c};-R_{2})f^{\ast}(bR_{1})\lambda_{R_{1}}\nonumber\\
&  \quad+e(-s/2)\sum\limits_{j=1}^{c}\sum\limits_{\mu=0}^{k-1}\sum
\limits_{v=0}^{k-1}f(c(\left[  R_{2}+d(j-\left\{  R_{1}\right\}  )/c\right]
-v))f^{\ast}(b(\mu c+j+\left[  R_{1}\right]  )) I(z,s,c,d,r_{1},r_{2}%
)\nonumber
\end{align}
where $L(s;A_{c};R_{2})$ is given by (\ref{12}) and
\begin{equation}
I(z,s,c,d,r_{1},r_{2}) =\int\limits_{C}u^{s-1}\frac{\exp(-\left(
(c\mu+j-\left\{  R_{1}\right\}  )/ck\right)  (cz+d)ku)}{\exp(-ku(cz+d))-1}%
\frac{\exp((\left(  v+\left\{  (dj+\rho)/c\right\}  \right)  /k)ku)}%
{\exp(ku)-1}du.\label{int}%
\end{equation}
Here, we choose the branch of $u^{s}$ with $0<\arg u<2\pi$. Also, $C$ is a
loop beginning at $+\infty$, proceeding in the upper half-plane, encircling
the origin in the positive direction so that $u=0$ is the only zero of
$(\exp(-ku(cz+d))-1)(\exp\left(  ku\right)  -1)$ lying "inside" the loop, and
then returning to $+\infty$ in the lower half-plane. Secondly, if $b\equiv
c\equiv0\left(  \operatorname{mod}k\right)  ,$ then for $z\in\mathbb{K}$ and
all $s,$%
\begin{align}
&  (cz+d)^{-s}\Gamma(s)G(Vz,s;A,B;r_{1},r_{2})\label{13}\\
&  =\Gamma(s)G(z,s;A_{d},B_{a};R_{1},R_{2})-2i\Gamma(s)\sin(\pi s)L(s;B_{-a}%
;-R_{2})f(-dR_{1})\lambda_{R_{1}}\nonumber\\
&  \quad+e(-s/2)\sum\limits_{j=1}^{c}\sum\limits_{\mu=0}^{k-1}\sum
\limits_{v=0}^{k-1}f^{\ast}(-a(\left[  R_{2}+d(j-\left\{  R_{1}\right\}
)/c\right]  -v+d\mu))f(-d(j+\left[  R_{1}\right]  )) I(z,s,c,d,r_{1}%
,r_{2}).\nonumber
\end{align}

\end{theorem}

\begin{remark}
Observe that the transformation formulas obtained by Berndt \cite{12,2,7} are
the special cases of Theorem \ref{mainteo}.
\end{remark}

\begin{proof}
For $z\in\mathbb{H}$ and $\operatorname{Re}\left(  s\right)  >2,$
\[
G(Vz,s;A,B;r_{1},r_{2})=\sum\limits_{m,n=-\infty}^{\infty}f(m)f^{\ast
}(n)\left\{  \frac{((M+R_{1})z+N+R_{2})}{cz+d}\right\}  ^{-s}%
\]
where $M=ma+nc,$ $N=mb+nd$. $M$ and $N$ range over all pairs of integers
except for the possibility $-R_{1},-R_{2}$ as the pair $m,$ $n$ ranges over
all pairs of integers, except for possibly the pair $-r_{1},-r_{2}$ (since
$ad-bc=1$). Thus,
\begin{align*}
&  G(Vz,s;A,B;r_{1},r_{2})\\
&  =\sum\limits_{M,N=-\infty}^{\infty}f(Md-Nc)f^{\ast}(Na-Mb)\left\{
\frac{((M+R_{1})z+N+R_{2})}{cz+d}\right\}  ^{-s}\\
&  =\sum\limits_{m,n=-\infty}^{\infty}f(-nc)f^{\ast}(-mb)\left\{
\frac{((m+R_{1})z+n+R_{2})}{cz+d}\right\}  ^{-s},\text{ }a\equiv
d\equiv0(\operatorname{mod}k),\\
&  =\sum\limits_{m,n=-\infty}^{\infty}f(dm)f^{\ast}(an)\left\{  \frac
{((m+R_{1})z+n+R_{2})}{cz+d}\right\}  ^{-s},\text{ }b\equiv c\equiv
0(\operatorname{mod}k).
\end{align*}
Using Lemma \ref{le1}, we obtain that%
\begin{align}
(cz+d)^{-s}G(Vz,s;A,B;r_{1},r_{2}) &  =\left(  e(-s)\sum
\limits_{\overset{m+R_{1}\leq0}{d(m+R_{1})>c(n+}R_{2})}+\sum
\limits_{\substack{m,n \\\text{otherwise}}}\right)  \frac{f^{\ast}%
(-mb)f(-nc)}{((m+R_{1})z+n+R_{2})^{s}}\nonumber\\
&  =G(z,s;B_{-b},A_{-c};R_{1},R_{2})+\left(  e(-s)-1\right)  g\left(
z,s;B_{-b},A_{-c};R_{1},R_{2}\right) \label{5}%
\end{align}
where
\begin{equation}
g\left(  z,s;B_{-b},A_{-c};R_{1},R_{2}\right)  =\sum\limits_{\overset{m+R_{1}%
\leq0}{d(m+R_{1})>c(n+}R_{2})}\frac{f^{\ast}(-mb)f(-nc)}{((m+R_{1}%
)z+n+R_{2})^{s}}.\label{6a}%
\end{equation}
In (\ref{6a}), we replace $m$ by $-m$ and $n$ by $-n$ and separate the terms
with $m=R_{1}.$ Thus%
\begin{equation}
g\left(  z,s;B_{-b},A_{-c};R_{1},R_{2}\right)  =e(s/2)\left\{  \lambda_{R_{1}%
}f^{\ast}(bR_{1})L(s,A_{c};-R_{2})+h(z,s;B_{b},A_{c};-R_{1},-R_{2})\right\}
\label{6}%
\end{equation}
where
\[
h(z,s;B_{b},A_{c};-R_{1},-R_{2})=\sum\limits_{m>R_{1}}\sum\limits_{n>R_{2}%
+d(m-R_{1})/c}\frac{f^{\ast}(mb)f(nc)}{((m-R_{1})z+n-R_{2})^{s}}.
\]
Since$\ \operatorname{Re}((m-R_{1})z+n-R_{2})>0$ for $x>-d/c,$ using the
Euler's integral representation of $\Gamma(s),$ we find for $z\in\mathbb{K}$
and $\operatorname{Re}\left(  s\right)  >2$ that%
\[
\Gamma(s)h(z,s;B_{b},A_{c};-R_{1},-R_{2})=\sum\limits_{m>R_{1}}\sum
\limits_{n>R_{2}+d(m-R_{1})/c}f(nc)f^{\ast}(mb)\int\limits_{0}^{\infty}%
u^{s-1}\exp(-(m-R_{1})zu-(n-R_{2})u)du.
\]
We set $m=m^{\prime}c+j+\left[  R_{1}\right]  +1,$ $0\leq m^{\prime}<\infty, $
$0\leq j\leq c-1$ and $n=n^{\prime}+\left[  R_{2}+d(m-R_{1})/c\right]  +1. $
The double sum above becomes%
\begin{align*}
&  \sum\limits_{j=0}^{c-1}\sum\limits_{m^{\prime}=0}^{\infty}\sum
\limits_{n^{\prime}=0}^{\infty}f^{\ast}(b(m^{\prime}c+j+\left[  R_{1}\right]
+1))f(c(n^{\prime}+\left[  R_{2}+d(m^{\prime}c+j-\left\{  R_{1}\right\}
+1)/c\right]  +1))\\
&  \times\int\limits_{0}^{\infty}u^{s-1}\exp(-(m^{\prime}c+j-\left\{
R_{1}\right\}  +1)zu)\exp(-(n^{\prime}+\left[  R_{2}+d(m^{\prime}c+j-\left\{
R_{1}\right\}  +1)/c\right]  +1-R_{2})u)du
\end{align*}
Replacing $j+1$ by $j,$ and using that $d\equiv0(\operatorname{mod}$ $k),$ put
$m^{\prime}=mk+\mu,$ $0\leq m<\infty,$ $0\leq\mu\leq k-1,$ and put $n^{\prime
}=nk+v,$ $0\leq n<\infty,$ $0\leq v\leq k-1$, we have
\begin{align}
&  \Gamma(s)h(z,s;B_{b},A_{c};-R_{1},-R_{2})\nonumber\\
&  =\sum\limits_{j=1}^{c}\sum\limits_{\mu=0}^{k-1}\sum\limits_{v=0}%
^{k-1}f^{\ast}(b(\mu c+j+\left[  R_{1}\right]  ))f(c(v+\left[  R_{2}%
+d(j-\left\{  R_{1}\right\}  )/c\right]  +1))\nonumber\\
&  \quad\times\int\limits_{0}^{\infty}u^{s-1}\exp(-(c\mu+j-\left\{
R_{1}\right\}  )zu-(v+\left[  R_{2}+d(j-\left\{  R_{1}\right\}  )/c\right]
+1-R_{2})u)\nonumber\\
&  \quad\times\sum\limits_{m=0}^{\infty}\sum\limits_{n=0}^{\infty}%
\exp(-mku\left(  cz+d\right)  -nku)du\nonumber\\
&  =\sum\limits_{j=1}^{c}\sum\limits_{\mu=0}^{k-1}\sum\limits_{v=0}%
^{k-1}f^{\ast}(b(\mu c+j+\left[  R_{1}\right]  ))f(c(v+\left[  R_{2}%
+d(j-\left\{  R_{1}\right\}  )/c\right]  +1))\nonumber\\
&  \times\int\limits_{0}^{\infty}u^{s-1}\frac{\exp\left(  -(c\mu+j-\left\{
R_{1}\right\}  )zu\right)  }{1-\exp(-ku(cz+d))}\frac{\exp\left(  -\left(
v+1-R_{2}+d\mu+\left[  R_{2}+d(j-\left\{  R_{1}\right\}  )/c\right]  \right)
u\right)  }{1-\exp(-ku)}du\nonumber\\
&  =-\sum\limits_{j=1}^{c}\sum\limits_{\mu=0}^{k-1}\sum\limits_{v=0}%
^{k-1}f^{\ast}(b(\mu c+j+\left[  R_{1}\right]  ))f(c(\left[  R_{2}%
+d(j-\left\{  R_{1}\right\}  )/c\right]  -v))\nonumber\\
&  \quad\times\int\limits_{0}^{\infty}u^{s-1}\frac{\exp(-(c\mu+j-\left\{
R_{1}\right\}  )(cz+d)ku/ck)}{\exp(-ku(cz+d))-1}\frac{\exp(((v+\left\{
(dj+\rho)/c\right\}  )/k)ku)}{\exp(ku)-1}du\nonumber\\
&  =\sum\limits_{j=1}^{c}\sum\limits_{\mu=0}^{k-1}\sum\limits_{v=0}%
^{k-1}f^{\ast}(b(\mu c+j+\left[  R_{1}\right]  ))f(c(\left[  R_{2}%
+d(j-\left\{  R_{1}\right\}  )/c\right]  -v))\frac{I(z,s,c,d,r_{1},r_{2}%
)}{1-e(s)}\label{7}%
\end{align}
where
\[
I(z,s,c,d,r_{1},r_{2})=\int\limits_{C}u^{s-1}\frac{\exp(-\left(
(c\mu+j-\left\{  R_{1}\right\}  )/ck\right)  (cz+d)ku)}{\exp(-ku(cz+d))-1}%
\frac{\exp(((v+\left\{  (dj+\rho)/c\right\}  )/k)ku)}{\exp(ku)-1}du.
\]
Here, in the next to the last step, we have multiplied the numerator and
denominator by $\exp(ku)$ and then replaced $k-1-v$ by $v.$ In the last step,
we have used a classical method of Riemann to convert the integral over
$(0,\infty)$ to a loop integral \cite{21}. Combining (\ref{5}), (\ref{6}) and
(\ref{7}) we deduce (\ref{13-a}). By analytic continuation the result is valid
for all $s$. The proof of (\ref{13}) is completely analogous.
\end{proof}

We will also need the following theorem whose proof is similar to the proof of
(\ref{13-a}).

\begin{theorem}
Under the conditions of Theorem \ref{mainteo}, for $a\equiv d\equiv0\left(
\operatorname{mod}k\right)  $ we have%
\begin{align}
&  (cz+d)^{-s}\Gamma(s)G(Vz,s;B_{-\beta},A_{-\alpha};r_{1},r_{2})\label{34}\\
&  =\Gamma(s)G(z,s;A_{\alpha b},B_{\beta c};R_{1},R_{2})-2i\Gamma(s)\sin(\pi
s)f(-\alpha bR_{1})L(s,B_{-\beta c};-R_{2})\nonumber\\
&  +e(-s/2)\sum\limits_{j=1}^{c}\sum\limits_{\mu=0}^{k-1}\sum\limits_{v=0}%
^{k-1}f(-\alpha b(\mu c+j+\left[  R_{1}\right]  ))f^{\ast}(-\beta c(\left[
R_{2}+d(j-\left\{  R_{1}\right\}  )/c\right]  -v))I(z,s,c,d,r_{1}%
,r_{2}),\nonumber
\end{align}
where $I(z,s,c,d,r_{1},r_{2})$ is given by (\ref{int}).
\end{theorem}

Additionally, setting $a=d=0$ and $c=-b=1$ in (\ref{34}), we have%
\begin{align}
&  z^{-s}\Gamma(s)G(-1/z,s;B_{-\beta},A_{-\alpha};r_{1},r_{2})\nonumber\\
&  =\Gamma(s)G(z,s;A_{-\alpha},B_{\beta};R_{1},R_{2})-2i\Gamma(s)\sin(\pi
s)f(\alpha R_{1})L(s,B_{-\beta};-R_{2})\nonumber\\
&  \quad+e(-s/2)\sum\limits_{\mu=0}^{k-1}\sum\limits_{v=0}^{k-1}f(\alpha
(\mu+1+\left[  R_{1}\right]  ))f^{\ast}(-\beta(\left[  R_{2}\right]
-v))I(z,s,1,0,r_{1},r_{2}).\label{14}%
\end{align}

\section{The periodic analogue of Dedekind sum}

Our main results can be simplified when $s$ is an integer. In this case,
$I(z,s,c,d,r_{1},r_{2})$ can be evaluated by the residue theorem with the aid
of (\ref{8}). Therefore, if $s=-N,$ where $N$ is a nonnegative integer, a
simple calculation yields%
\begin{equation}
I(z,-N,c,d,r_{1},r_{2}) =2\pi ik^{N}\sum\limits_{m+n=N+2}B_{m}\left(
\frac{c\mu+j-\left\{  R_{1}\right\}  }{ck}\right)  B_{n}\left(  \frac
{v+\left\{  (dj+\rho)/c\right\}  }{k}\right)  \frac{\left(  -(cz+d)\right)
^{m-1}}{m!n!}.\label{int1}%
\end{equation}
From which it is seen for $s=-N$ that Theorem \ref{mainteo} is valid for
$z\in\mathbb{H}$ by analytic continuation.

\subsection{The case $s=r_{1}=r_{2}=0$}

For $s=r_{1}=r_{2}=0,$ (\ref{int1}) becomes
\[
I(z,0,c,d,0,0) =-\frac{\pi i}{cz+d}B_{2}\left(  \frac{v+\left\{  dj/c\right\}
}{k}\right)  -\pi i(cz+d)B_{2}\left(  \frac{c\mu+j}{ck}\right)  +2\pi
iB_{1}\left(  \frac{c\mu+j}{ck}\right)  B_{1}\left(  \frac{v+\left\{
dj/c\right\}  }{k}\right)  .
\]

To begin with, let us consider the case $a\equiv d\equiv0\left(
\operatorname{mod}k\right)  .$ Then, (\ref{13-a}) can be written as%
\begin{align}
&  \lim_{s\rightarrow0}\Gamma(s)\left(  \frac{1}{\left(  cz+d\right)  ^{s}%
}G(Vz,s;A,B)-G(z,s;B_{-b},A_{-c})\right) \nonumber\\
&  =-2if^{\ast}(0)\lim_{s\rightarrow0}\Gamma(s)\sin(\pi s)L(s,A_{c}%
;0)+\sum\limits_{j=1}^{c}\sum\limits_{\mu=0}^{k-1}\sum\limits_{v=0}%
^{k-1}f^{\ast}(b(\mu c+j))f(c(\left[  dj/c\right]
-v))I(z,0,c,d,0,0).\label{9}%
\end{align}
We shall evaluate the triple sums in (\ref{9}). Let%
\begin{align*}
&  \sum\limits_{j=1}^{c}\sum\limits_{\mu=0}^{k-1}\sum\limits_{v=0}%
^{k-1}f^{\ast}\left(  b(\mu c+j)\right)  f\left(  c(\left[  d(j)/c\right]
-v)\right)  \left(  -\frac{\pi i}{cz+d}B_{2}\left(  \frac{v+\left\{
dj/c\right\}  }{k}\right)  \right. \\
&  \quad\left.  -\pi i(cz+d)B_{2}\left(  \frac{c\mu+j}{ck}\right)  +2\pi
iB_{1}\left(  \frac{c\mu+j}{ck}\right)  B_{1}\left(  \frac{v+\left\{
dj/c\right\}  }{k}\right)  \right) \\
&  =T_{1}+T_{2}+T_{3}.
\end{align*}
Firstly,
\begin{align*}
T_{1}  &  =-\pi i\frac{1}{cz+d}\sum\limits_{j=1}^{c}\sum\limits_{\mu=0}%
^{k-1}f^{\ast}(b(\mu c+j))\sum\limits_{v=0}^{k-1}f(c(\left[  dj/c\right]
-v))B_{2}\left(  \frac{v+\left\{  dj/c\right\}  }{k}\right) \\
&  =-\pi i\frac{2}{cz+d}kB_{0}(B)\sum\limits_{j=1}^{c}\sum\limits_{v=0}%
^{k-1}f(c(\left[  dj/c\right]  -v))P_{2}\left(  \frac{v+\left\{  dj/c\right\}
}{k}\right) \\
&  =-\pi i\frac{2}{cz+d}kB_{0}(B)\sum\limits_{j=1}^{c}\sum\limits_{v=0}%
^{k-1}f(-cv)P_{2}\left(  \frac{v+dj/c}{k}\right)  .
\end{align*}
For the convenience with the notation $P_{r}(x,A),$ denote the sum over $v$ as
$P_{r}(x,A_{c}),$ namely,%
\begin{equation}
k^{r-1}\sum\limits_{v=0}^{k-1}f(-cv)P_{r}\left(  \frac{v+x}{k}\right)
=P_{r}(x,A_{c}).\label{pr*}%
\end{equation}
Note that $P_{r}(x,A_{1})=P_{r}(x,A).$ Thus, for $\left(  d,c\right)  =1$ and
$d\equiv0\left(  \operatorname{mod}k\right)  $,
\begin{align*}
\sum\limits_{j=1}^{c}P_{r}\left(  dj/c,A_{c}\right)   &  =k^{r-1}%
\sum\limits_{v=0}^{k-1}f(-cv)\sum\limits_{j=1}^{c}P_{r}\left(  \frac
{v+dj/c}{k}\right) \\
&  =k^{r-1}\sum\limits_{v=0}^{k-1}f(-cv)\sum\limits_{j=1}^{c}P_{r}\left(
\frac{v}{k}+\frac{mj}{c}\right) \\
&  =k^{r-1}c^{1-r}\sum\limits_{v=0}^{k-1}f(-cv)P_{r}\left(  \frac{cv}%
{k}\right) \\
&  =c^{1-r}P_{r}\left(  0,A\right)  .
\end{align*}
Therefore, with the use of $P_{r}\left(  0,A\right)  =(-1)^{r}B_{r}(A)/r!,$
for $r\geq2$ or $r=0$ \cite[Eq. (2.12)]{10},
\[
T_{1}=-\frac{\pi i}{c\left(  cz+d\right)  }B_{0}(B)B_{2}(A).
\]
Secondly, using (\ref{10}) and (\ref{26})
\begin{align*}
T_{2}  &  =-\pi i(cz+d)\sum\limits_{j=1}^{c}\sum\limits_{\mu=0}^{k-1}f^{\ast
}(b(\mu c+j))B_{2}\left(  \frac{c\mu+j}{ck}\right)  \sum\limits_{v=0}%
^{k-1}f(c(\left[  dj/c\right]  -v))\\
&  =-2\pi i(cz+d)kB_{0}(A)\sum\limits_{n=0}^{ck-1}f^{\ast}(-bn)P_{2}\left(
\frac{n}{ck}\right) \\
&  =-2\pi i(cz+d)kB_{0}(A)\sum\limits_{m=0}^{k-1}\sum\limits_{v=0}%
^{c-1}f^{\ast}(-bm)P_{2}\left(  \frac{vk+m}{ck}\right) \\
&  =-\frac{2\pi i}{c}(cz+d)B_{0}(A)P_{2}(0,B_{b}).
\end{align*}
Finally,
\begin{align*}
&  (2\pi i)^{-1}T_{3}\\
&  =\sum\limits_{j=1}^{c}\sum\limits_{\mu=0}^{k-1}\sum\limits_{v=0}%
^{k-1}f^{\ast}(b(\mu c+j))f(c(\left[  dj/c\right]  -v))B_{1}\left(  \frac
{c\mu+j}{ck}\right)  B_{1}\left(  \frac{v+\left\{  dj/c\right\}  }{k}\right)
\\
&  =\sum\limits_{j=1}^{c-1}\sum\limits_{\mu=0}^{k-1}f^{\ast}(b(\mu
c+j))P_{1}\left(  \frac{c\mu+j}{ck}\right)  \sum\limits_{v=0}^{k-1}f(c(\left[
dj/c\right]  -v))P_{1}\left(  \frac{v-\left[  dj/c\right]  +dj/c}{k}\right) \\
&  \quad+\sum\limits_{\mu=0}^{k-1}f^{\ast}(bc(\mu+1))B_{1}\left(  \frac{\mu
+1}{k}\right)  \sum\limits_{v=0}^{k-1}f(-cv)P_{1}\left(  \frac{v}{k}\right) \\
&  =\sum\limits_{\mu=0}^{k-1}\sum\limits_{j=1}^{c}f^{\ast}(b(\mu
c+j))P_{1}\left(  \frac{c\mu+j}{ck}\right)  \sum\limits_{v=0}^{k-1}%
f(-cv)P_{1}\left(  \frac{v+dj/c}{k}\right) \\
&  \quad+P_{1}\left(  0,A_{c}\right)  \sum\limits_{\mu=1}^{k}f^{\ast}%
(bc\mu)\left\{  B_{1}\left(  \frac{\mu}{k}\right)  -P_{1}\left(  \frac{\mu}%
{k}\right)  \right\}  .
\end{align*}
So,
\[
T_{3}=2\pi i\sum\limits_{n=1}^{ck}f^{\ast}(bn)P_{1}\left(  \frac{n}%
{ck}\right)  P_{1}\left(  \frac{dn}{c},A_{c}\right)  +2\pi if^{\ast}%
(0)P_{1}\left(  0,A_{c}\right)  .
\]
Then, we find that
\begin{align}
T_{1}+T_{2}+T_{3}  &  =-\frac{\pi i}{c(cz+d)}B_{0}(B)B_{2}(A)-\frac{2\pi i}%
{c}(cz+d)B_{0}(A)P_{2}(0,B_{b})\nonumber\\
&  \quad+2\pi i\sum\limits_{n=1}^{ck}f^{\ast}(bn)P_{1}\left(  \frac{n}%
{ck}\right)  P_{1}\left(  \frac{dn}{c},A_{c}\right)  +2\pi if^{\ast}%
(0)P_{1}\left(  0,A_{c}\right)  .\label{15}%
\end{align}

\begin{definition}
Let $c$ and $d$ be coprime integers with $d\equiv0\left(  \operatorname{mod}%
k\right)  $ and $c>0$. For $bc\equiv-1$ $\left(  \operatorname{mod}d\right)
,$ the periodic Dedekind sum $s\left(  d,c;B_{b},A_{c}\right)  $ is defined
by
\[
s\left(  d,c;B_{b},A_{c}\right)  =\sum\limits_{n=1}^{ck}f^{\ast}%
(bn)P_{1}\left(  \frac{n}{ck}\right)  P_{1}\left(  \frac{dn}{c},A_{c}\right)
.
\]

\end{definition}

We emphasize that the sum $s\left(  d,c;B,A\right)  $ has been defined by
Berndt \cite[Section 7]{15} without restrictions $ad-bc=1$ and $a\equiv
d\equiv0\left(  \operatorname{mod}k\right)  $.

To compute $L(0;A_{\beta};\theta)$ we utilize (\ref{Ls}), the formula
$\zeta\left(  0,\theta\right)  =1/2-\theta=-B_{1}\left(  \theta\right)  ,$
$0<\theta\leq1,$ and $B_{1}\left(  1-\theta\right)  =-B_{1}\left(
\theta\right)  .$ We then see that
\begin{align}
L(0;A_{\beta};\theta)  &  =-\sum\limits_{j=0}^{k-1}f(\beta\left(  j-\left[
\theta\right]  +\lambda_{\theta}\right)  )B_{1}\left(  \frac{j+\left\{
\theta\right\}  +\lambda_{\theta}}{k}\right) \nonumber\\
&  =-\sum\limits_{j=0}^{k-2}f(\beta\left(  j-\left[  \theta\right]
+\lambda_{\theta}\right)  )P_{1}\left(  \frac{j+\left\{  \theta\right\}
+\lambda_{\theta}}{k}\right)  -f(\beta\left(  -1-\left[  \theta\right]
+\lambda_{\theta}\right)  )B_{1}\left(  1-\frac{1-\left\{  \theta\right\}
-\lambda_{\theta}}{k}\right) \nonumber\\
&  =\sum\limits_{j=1}^{k-1}f(-\beta\left(  j+1+\left[  \theta\right]
-\lambda_{\theta}\right)  )P_{1}\left(  \frac{j+1-\left\{  \theta\right\}
-\lambda_{\theta}}{k}\right)  +f(-\beta\left(  1+\left[  \theta\right]
-\lambda_{\theta}\right)  )P_{1}\left(  \frac{1-\left\{  \theta\right\}
-\lambda_{\theta}}{k}\right) \nonumber\\
&  =\sum\limits_{j=0}^{k-1}f(-\beta j)P_{1}\left(  \frac{j-\theta}{k}\right)
=P_{1}\left(  -\theta,A_{\beta}\right)  .\label{L0}%
\end{align}
Therefore, we have
\begin{equation}
-2if^{\ast}(0)\lim_{s\rightarrow0}s\Gamma(s)\frac{\sin(\pi s)}{s\pi}\pi
L(s;A_{c};0)=-2i\pi f^{\ast}(0)P_{1}\left(  0,A_{c}\right)  .\label{16}%
\end{equation}
Next, substituting (\ref{15}) and (\ref{16}) in (\ref{9}) gives%
\begin{align}
&  \lim_{s\rightarrow0}\Gamma(s)\left(  (cz+d)^{-s}G(Vz,s;A,B)-G(z,s;B_{-b}%
,A_{-c})\right) \label{17a}\\
&  \ =2\pi is\left(  d,c;B_{b},A_{c}\right)  -\frac{\pi i}{c(cz+d)}%
B_{0}(B)B_{2}(A)-\frac{2\pi i}{c}(cz+d)B_{0}(A)P_{2}\left(  0,B_{b}\right)
.\nonumber
\end{align}

The second case we will consider is $b\equiv c\equiv0\left(
\operatorname{mod}k\right)  .$ Now\textbf{\ }(\ref{13}) becomes\textbf{\ }%
\begin{align}
&  \lim_{s\rightarrow0}\Gamma(s)\left(  (cz+d)^{-s}G(Vz,s;A,B)-G(z,s;A_{d}%
,B_{a})\right) \nonumber\\
&  =-2if^{\ast}(0)\lim_{s\rightarrow0}\Gamma(s)\sin(\pi s)L(s;B_{-a}%
;0)\nonumber\\
&  \quad+\sum\limits_{j=1}^{c}\sum\limits_{\mu=0}^{k-1}\sum\limits_{v=0}%
^{k-1}f^{\ast}(-a(\left[  dj/c\right]  +v+d\mu))f(-d(\mu
c+j))I(z,0,c,d,0,0).\label{18}%
\end{align}
Let%
\begin{align*}
&  \sum\limits_{j=1}^{c}\sum\limits_{\mu=0}^{k-1}\sum\limits_{v=0}%
^{k-1}f(-dj))f^{\ast}(-a(\left[  dj/c\right]  -v+d\mu))\left(  -\pi
i(cz+d)B_{2}\left(  \frac{c\mu+j}{ck}\right)  \right. \\
&  \quad\left.  -\frac{\pi i}{cz+d}B_{2}\left(  \frac{v+\left\{  dj/c\right\}
}{k}\right)  +2\pi iB_{1}\left(  \frac{c\mu+j}{ck}\right)  B_{1}\left(
\frac{v+\left\{  dj/c\right\}  }{k}\right)  \right) \\
&  =T_{4}+T_{5}+T_{6}.
\end{align*}
Similar to $T_{1},$ $T_{2}$ and $T_{3}$, it can be obtained that\
\begin{align*}
T_{4}  &  =-\frac{2\pi i}{c}(cz+d)B_{0}(B)P_{2}(0,A_{d}),\\
T_{5}  &  =-\frac{2\pi i}{c\left(  cz+d\right)  }B_{0}(B)P_{2}\left(
0,A\right)  ,\\
T_{6}  &  =2\pi i\sum\limits_{n=1}^{ck}f(-dn)P_{1}\left(  \frac{n}{ck}\right)
P_{1}\left(  \frac{dn}{c},B_{-a}\right)  +2\pi if(0)P_{1}\left(
0,B_{-a}\right)  .
\end{align*}

\begin{definition}
Let $c$ and $d$ be coprime integers with $c\equiv0\left(  \operatorname{mod}%
k\right)  $ and $c>0$. For $ad\equiv1$ $\left(  \operatorname{mod}c\right)  ,$
the periodic Dedekind sum is defined by\textbf{\ }%
\[
s\left(  d,c;A_{d},B_{a}\right)  =\sum\limits_{n=1}^{ck}f(dn)P_{1}\left(
\dfrac{n}{ck}\right)  P_{1}\left(  \dfrac{dn}{c},B_{a}\right)  .
\]

\end{definition}

Consequently, we have%
\begin{align*}
T_{4}+T_{5}+T_{6}  &  =2\pi is\left(  d,c;A_{-d},B_{-a}\right)  -\frac{2\pi
i}{c}(cz+d)B_{0}(B)P_{2}(0,A_{d})\\
&  \quad-\frac{\pi i}{c(cz+d)}B_{0}(B)B_{2}\left(  A\right)  +2\pi
if(0)P_{1}\left(  0,B_{-a}\right)  .
\end{align*}
From (\ref{16}), the following equality holds%
\begin{equation}
-2if(0)\lim_{s\rightarrow0}\Gamma(s)\sin(\pi s)L(s;B_{-a};0)=-2i\pi
f(0)P_{1}\left(  0,B_{-a}\right)  .\label{25}%
\end{equation}
Substituting these in (\ref{18}) gives
\begin{align}
&  \lim_{s\rightarrow0}\Gamma(s)\left(  (cz+d)^{-s}G(Vz,s;A,B)-G(z,s;A_{d}%
,B_{a})\right) \label{19a}\\
&  =2\pi is\left(  d,c;A_{-d},B_{-a}\right)  -2\pi i\frac{cz+d}{c}%
B_{0}(B)P_{2}(0,A_{d})-\frac{\pi i}{c(cz+d)}B_{0}(B)B_{2}(A).\nonumber
\end{align}
We summarize the results obtained above in the next theorem.

\begin{theorem}
\label{d1}Let $z\in\mathbb{H}$. If\textbf{\ }$a\equiv d\equiv0\left(
\operatorname{mod}k\right)  ,$ then%
\begin{align}
&  \lim_{s\rightarrow0}\Gamma(s)\left(  (cz+d)^{-s}G(Vz,s;A,B)-G(z,s;B_{-b}%
,A_{-c})\right) \label{17}\\
&  =2\pi is\left(  d,c;B_{b},A_{c}\right)  -\frac{\pi i}{c(cz+d)}B_{0}%
(B)B_{2}(A)-2\pi i\frac{cz+d}{c}B_{0}(A)P_{2}\left(  0,B_{b}\right)
.\nonumber
\end{align}
If $b\equiv c\equiv0\left(  \operatorname{mod}k\right)  ,$ then%
\begin{align}
&  \lim_{s\rightarrow0}\Gamma(s)\left(  (cz+d)^{-s}G(Vz,s;A,B)-G(z,s;A_{d}%
,B_{a})\right) \label{19}\\
&  =2\pi is\left(  d,c;A_{-d},B_{-a}\right)  -\frac{\pi i}{c(cz+d)}%
B_{0}(B)B_{2}(A)-2\pi i\frac{cz+d}{c}B_{0}(B)P_{2}(0,A_{d}).\nonumber
\end{align}

\end{theorem}

Now we are ready to prove a reciprocity formula for $s\left(  d,c;B_{b}%
,A_{c}\right)  $ analogous to that of the previous Dedekind sums.

\begin{theorem}
\label{rep1}Let $c$ and $d$ be coprime positive integers with $d\equiv0\left(
\operatorname{mod}k\right)  $. For $bc\equiv-1$ $\left(  \operatorname{mod}%
d\right)  ,$
\begin{align*}
&  s\left(  -c,d;A_{c},B_{-b}\right)  -s\left(  d,c;B_{b},A_{c}\right) \\
&  =P_{1}\left(  0,B_{-b}\right)  P_{1}\left(  0,A_{-c}\right)  -\frac{d}%
{c}B_{0}(A)P_{2}\left(  0,B_{b}\right)  -\frac{c}{d}B_{0}\left(  B\right)
P_{2}\left(  0,A_{c}\right)  -\frac{1}{2dc}B_{0}(B)B_{2}(A).
\end{align*}

\end{theorem}

\begin{proof}
Let $Tz=V\left(  -1/z\right)  =(bz-a)/(dz-c).$ Replacing $z$ by $-1/z$ in
(\ref{17}) gives\textbf{\ }%
\begin{align}
&  \lim_{s\rightarrow0}z^{s}\Gamma(s)\left(  \frac{1}{\left(  dz-c\right)
^{s}}G\left(  V\left(  -1/z\right)  ,s;A,B\right)  -z^{-s}G\left(
-1/z,s;B_{-b},A_{-c}\right)  \right) \nonumber\\
&  =2\pi is\left(  d,c;B_{b},A_{c}\right)  -\pi iz\frac{B_{0}(B)B_{2}%
(A)}{c\left(  dz-c\right)  }-2\pi i\frac{dz-c}{cz}B_{0}(A)P_{2}\left(
0,B_{b}\right) \label{23}%
\end{align}
and replacing $Vz$ by $Tz$ in (\ref{19}) gives
\begin{align}
&  \lim_{s\rightarrow0}\Gamma(s)\left(  \frac{1}{\left(  dz-c\right)  ^{s}%
}G(Tz,s;A,B)-G(z,s;A_{-c},B_{b})\right) \label{24}\\
&  =2\pi is\left(  -c,d;A_{c},B_{-b}\right)  -\pi i\frac{B_{0}(B)B_{2}%
(A)}{d\left(  dz-c\right)  }-\frac{2\pi i}{d}(dz-c)B_{0}(B)P_{2}\left(
0,A_{-c}\right)  .\nonumber
\end{align}
Lastly, the following limit needs to be computed\textbf{\ }
\[
\lim_{s\rightarrow0}\Gamma(s)\left(  z^{-s}G\left(  -1/z,s;B_{-b}%
,A_{-c}\right)  -G(z,s;A_{-c},B_{b})\right)  .
\]
For\ $s=r_{1}=r_{2}=0,$\textbf{\ }$\beta=b$ and $\alpha=c,$\ (\ref{14}) and
(\ref{int1}) become
\begin{align}
&  \lim_{s\rightarrow0}\Gamma(s)\left(  \frac{1}{z^{s}}G(-1/z,s;B_{-b}%
,A_{-c})-G(z,s;A_{-c},B_{b})\right) \label{20}\\
&  =-2if(0)\lim_{s\rightarrow0}\Gamma(s)\sin(\pi s)L(s;B_{-b};0)+\sum
\limits_{\mu=0}^{k-1}\sum\limits_{v=0}^{k-1}f(c(\mu+1))f^{\ast}%
(bv))I(z,0,1,0,0,0)\nonumber
\end{align}
and
\[
I(z,0,1,0,0,0)=-\frac{\pi i}{z}B_{2}\left(  \frac{v}{k}\right)  -\pi
izB_{2}\left(  \frac{\mu+1}{k}\right)  +2\pi iB_{1}\left(  \frac{\mu+1}%
{k}\right)  B_{1}\left(  \frac{v}{k}\right)  ,
\]
respectively. It is obvious that
\begin{align*}
\sum\limits_{\mu=0}^{k-1}f(c(\mu+1))\sum\limits_{v=0}^{k-1}f^{\ast}%
(bv)B_{2}\left(  \frac{v}{k}\right)   &  =2B_{0}(A)P_{2}\left(  0,B_{b}%
\right)  ,\\
\sum\limits_{v=0}^{k-1}f^{\ast}(bv)\sum\limits_{\mu=0}^{k-1}f(c(\mu
+1))B_{2}\left(  \frac{\mu+1}{k}\right)   &  =2B_{0}\left(  B\right)
P_{2}\left(  0,A_{c}\right)
\end{align*}
and
\[
\sum\limits_{v=0}^{k-1}f^{\ast}(bv)B_{1}\left(  \frac{v}{k}\right)
\sum\limits_{\mu=0}^{k-1}f(c(\mu+1))B_{1}\left(  \frac{\mu+1}{k}\right)
=P_{1}\left(  0,B_{-b}\right)  P_{1}\left(  0,A_{-c}\right)  +f\left(
0\right)  P_{1}\left(  0,B_{-b}\right)  .
\]
Then (\ref{20}) takes the form, with the help of (\ref{25}), that%
\begin{align}
&  \lim_{s\rightarrow0}\Gamma(s)\left(  \frac{1}{z^{s}}G(-1/z,s;B_{-b}%
,A_{-c},0,0)-G(z,s;A_{-c},B_{b})\right) \label{22}\\
&  =2\pi iP_{1}\left(  0,B_{-b}\right)  P_{1}\left(  0,A_{-c}\right)
-\frac{2\pi i}{z}B_{0}(A)P_{2}\left(  0,B_{b}\right)  -2\pi izB_{0}\left(
B\right)  P_{2}\left(  0,A_{c}\right)  .\nonumber
\end{align}
Thus, using the fact that%
\begin{align*}
&  \lim_{s\rightarrow0}z^{s}\Gamma(s)\left(  \frac{1}{\left(  dz-c\right)
^{s}}G\left(  Tz,s;A,B\right)  -G(z,s;A_{-c},B_{b})\right) \\
&  =\lim_{s\rightarrow0}z^{s}\Gamma(s)\left(  \frac{1}{\left(  dz-c\right)
^{s}}G(V\left(  -1/z\right)  ,s;A,B)-\frac{1}{z^{s}}G\left(  -1/z,s;B_{-b}%
,A_{-c}\right)  \right) \\
&  \quad+\lim_{s\rightarrow0}z^{s}\Gamma(s)\left(  \frac{1}{z^{s}}G\left(
-1/z,s;B_{-b},A_{-c}\right)  -G(z,s;A_{-c},B_{b})\right)  ,
\end{align*}
it follows from (\ref{23}), (\ref{24}) and (\ref{22}) that%
\begin{align*}
&  s\left(  -c,d;A_{c},B_{-b}\right)  -s\left(  d,c;B_{b},A_{c}\right) \\
&  =P_{1}\left(  0,B_{-b}\right)  P_{1}\left(  0,A_{-c}\right)  -\frac{1}%
{2dc}B_{0}(B)B_{2}(A)-\frac{1}{z}B_{0}(A)P_{2}\left(  0,B_{b}\right) \\
&  \quad-zB_{0}\left(  B\right)  P_{2}\left(  0,A_{c}\right)  -\frac{1}%
{cz}\left(  dz-c\right)  B_{0}(A)P_{2}\left(  0,B_{b}\right)  +\frac{1}%
{d}(dz-c)B_{0}(B)P_{2}\left(  0,A_{-c}\right)  .
\end{align*}
From this, the reciprocity formula follows by setting $z=c/d$.
\end{proof}

\begin{remark}
The reciprocity formula may be stated as
\begin{align*}
&  s\left(  c,d;A_{c},B_{b}\right)  +s\left(  d,c;B_{b},A_{c}\right) \\
&  =-P_{1}\left(  0,B_{-b}\right)  P_{1}\left(  0,A_{-c}\right)  +\frac
{1}{2dc}B_{0}(B)B_{2}(A)+\frac{d}{c}B_{0}(A)P_{2}\left(  0,B_{b}\right)
+\frac{c}{d}B_{0}\left(  B\right)  P_{2}\left(  0,A_{c}\right)  -f(0)P_{1}%
\left(  0,B\right)  .
\end{align*}
Indeed, if we use that for $x\not \in \mathbb{Z},$
\[
P_{1}\left(  -x,B_{-b}\right)  =\sum\limits_{v=0}^{k-1}f^{\ast}(-bv)P_{1}%
\left(  -\frac{v+x}{k}\right)  =-P_{1}\left(  x,B_{b}\right)
\]
and for $x\in\mathbb{Z}$,
\begin{align*}
P_{1}\left(  -x,B_{-b}\right)   &  =\sum\limits_{\substack{v=0 \\v\not \equiv
x\left(  \operatorname{mod}k\right)  }}^{k-1}f^{\ast}(bv)P_{1}\left(
\frac{v-x}{k}\right)  +f^{\ast}(xb)P_{1}\left(  0\right) \\
&  =-\sum\limits_{\substack{v=0 \\v\not \equiv -x\left(  \operatorname{mod}%
k\right)  }}^{k-1}f^{\ast}(-bv)P_{1}\left(  \frac{v+x}{k}\right)  +f^{\ast
}(xb)P_{1}\left(  0\right) \\
&  =-\sum\limits_{v=0}^{k-1}f^{\ast}(-bv)P_{1}\left(  \frac{v+x}{k}\right)
+2f^{\ast}(xb)P_{1}\left(  0\right)  =-P_{1}\left(  x,B_{b}\right)  -f^{\ast
}(xb),
\end{align*}
then we deduce that%
\begin{equation}
s\left(  -c,d;A_{c},B_{-b}\right)  =-s\left(  c,d;A_{c},B_{b}\right)
-f(0)P_{1}\left(  0,B\right)  ,\label{51}%
\end{equation}
where we have used that $d\equiv0\left(  \operatorname{mod}k\right)  $ and
$bc\equiv-1\left(  \operatorname{mod}k\right)  .$
\end{remark}

\begin{remark}
Berndt \cite[Theorem 7.3]{15} has derived a reciprocity formula for the sum
$s\left(  c,d;A,B\right)  $ without restriction $d\equiv0\left(
\operatorname{mod}k\right)  $ by using Riemann--Stieltjes integral.
\end{remark}

\subsection{The case $s=0$, and $r_{1}$ and $r_{2}$ arbitrary}

Let $s=0,$ and $r_{1}$ and $r_{2}$\ arbitrary. From (\ref{int1}) we have%

\begin{align}
I(z,0,c,d,R_{1},R_{2})  &  =-\frac{\pi i}{cz+d}B_{2}\left(  \frac{v+\left\{
\left(  dj+\rho\right)  /c\right\}  }{k}\right)  -\pi i(cz+d)B_{2}\left(
\frac{c\mu+j-\left\{  R_{1}\right\}  }{ck}\right) \nonumber\\
&  \quad+2\pi iB_{1}\left(  \frac{c\mu+j-\left\{  R_{1}\right\}  }{ck}\right)
B_{1}\left(  \frac{v+\left\{  \left(  dj+\rho\right)  /c\right\}  }{k}\right)
.\label{40}%
\end{align}
Let us consider the case $a\equiv d\equiv0\left(  \operatorname{mod}k\right)
.$ Then, (\ref{13-a}) can be written as%
\begin{align}
&  \lim_{s\rightarrow0}\Gamma(s)\left(  \frac{1}{\left(  cz+d\right)  ^{s}%
}G(Vz,s;A,B;r_{1},r_{2})-G(z,s;B_{-b},A_{-c};R_{1},R_{2})\right) \nonumber\\
&  =-2if^{\ast}(bR_{1})\lambda_{R_{1}}\lim_{s\rightarrow0}\Gamma(s)\sin(\pi
s)L(s;A_{c};-R_{2})\label{36}\\
&  \quad+\sum\limits_{j=1}^{c}\sum\limits_{\mu=0}^{k-1}\sum\limits_{v=0}%
^{k-1}f^{\ast}(b(\mu c+j+\left[  R_{1}\right]  ))f(c(\left[  R_{2}+d\left(
j-\left\{  R_{1}\right\}  \right)  /c\right]  -v))I(z,0,c,d,R_{1}%
,R_{2}).\nonumber
\end{align}
Now we evaluate the triple sums in (\ref{36}). Let%
\begin{align*}
&  \sum\limits_{j=1}^{c}\sum\limits_{\mu=0}^{k-1}\sum\limits_{v=0}%
^{k-1}f^{\ast}(b(\mu c+j+\left[  R_{1}\right]  ))f(c(\left[  R_{2}+d\left(
j-\left\{  R_{1}\right\}  \right)  /c\right]  -v))\\
&  \quad\times\left(  -\frac{\pi i}{cz+d}B_{2}\left(  \frac{v+\left\{  \left(
dj+\rho\right)  /c\right\}  }{k}\right)  -\pi i(cz+d)B_{2}\left(  \frac
{c\mu+j-\left\{  R_{1}\right\}  }{ck}\right)  \right. \\
&  \qquad\ \left.  +2\pi iB_{1}\left(  \frac{c\mu+j-\left\{  R_{1}\right\}
}{ck}\right)  B_{1}\left(  \frac{v+\left\{  \left(  dj+\rho\right)
/c\right\}  }{k}\right)  \right) \\
&  \ =T_{7}+T_{8}+T_{9}.
\end{align*}
Since $\left(  dj+\rho\right)  /c=d\left(  j-\left\{  R_{1}\right\}  \right)
/c+R_{2}-\left[  R_{2}\right]  ,$ we conclude that%
\begin{align*}
T_{7}  &  =-\frac{\pi i}{cz+d}\sum\limits_{j=1}^{c}\sum\limits_{v=0}%
^{k-1}f(c(\left[  R_{2}+d\left(  j-\left\{  R_{1}\right\}  \right)  /c\right]
-v))B_{2}\left(  \frac{v+\left\{  \left(  dj+\rho\right)  /c\right\}  }%
{k}\right)  \sum\limits_{\mu=0}^{k-1}f^{\ast}(b(\mu c+j+\left[  R_{1}\right]
))\\
&  =-\frac{2\pi i}{cz+d}kB_{0}(B)\sum\limits_{j=1}^{c}\sum\limits_{v=0}%
^{k-1}f(-cv)P_{2}\left(  \frac{v+R_{2}}{k}+\frac{d\left(  j-\left\{
R_{1}\right\}  \right)  }{ck}\right)  .
\end{align*}
Thus, for $\left(  d,c\right)  =1$ and $d=mk$ (since $d\equiv0\left(
\operatorname{mod}k\right)  $) it follows from (\ref{26}) that
\begin{align*}
T_{7}  &  =-\frac{2\pi i}{c\left(  cz+d\right)  }kB_{0}(B)\sum\limits_{v=0}%
^{k-1}f(-cv)P_{2}\left(  \frac{cv+cR_{2}-mkR_{1}}{k}+m\left[  R_{1}\right]
\right) \\
&  =-\frac{2\pi i}{c\left(  cz+d\right)  }B_{0}(B)P_{2}\left(  cR_{2}%
-dR_{1},A\right)  .
\end{align*}
Again utilizing (\ref{26}) we have
\begin{align*}
T_{8}  &  =-\pi i(cz+d)\sum\limits_{j=1}^{c}\sum\limits_{\mu=0}^{k-1}f^{\ast
}(b(\mu c+j+\left[  R_{1}\right]  ))B_{2}\left(  \frac{c\mu+j-\left\{
R_{1}\right\}  }{ck}\right) \\
&  \qquad\times\sum\limits_{v=0}^{k-1}f((c\left[  R_{2}+d\left(  j-\left\{
R_{1}\right\}  \right)  /c\right]  -v))\\
&  =-2\pi i(cz+d)kB_{0}(A)\sum\limits_{n=1}^{ck}f^{\ast}(bn)P_{2}\left(
\frac{n-R_{1}}{ck}\right) \\
&  =-2\pi i(cz+d)kB_{0}(A)\sum\limits_{v=0}^{c-1}\sum\limits_{r=1}^{k}f^{\ast
}(br)P_{2}\left(  \frac{v}{c}+\frac{r-R_{1}}{ck}\right) \\
&  =-\frac{2\pi i}{c}(cz+d)B_{0}(A)P_{2}(R_{1},B_{b}).
\end{align*}
We now consider%
\begin{align*}
(2\pi i)^{-1}T_{9}  &  =\sum\limits_{j=1}^{c}\sum\limits_{\mu=0}^{k-1}%
\sum\limits_{v=0}^{k-1}f^{\ast}(b(\mu c+j+\left[  R_{1}\right]  ))f(c\left[
R_{2}+d\left(  j-\left\{  R_{1}\right\}  \right)  /c\right]  -v))\\
&  \quad\times B_{1}\left(  \frac{c\mu+j-\left\{  R_{1}\right\}  }{ck}\right)
B_{1}\left(  \frac{v+\left\{  \left(  dj+\rho\right)  /c\right\}  }{k}\right)
.
\end{align*}
We want to replace $B_{1}\left(  \frac{c\mu+j-\left\{  R_{1}\right\}  }%
{ck}\right)  $ by $P_{1}\left(  \frac{c\mu+j-\left\{  R_{1}\right\}  }%
{ck}\right)  ,$ this is valid except for $R_{1}\in%
\mathbb{Z}
,$ $\mu=k-1$ and $j=c.$ If $R_{1}\in%
\mathbb{Z}
,$ firstly separate the term $j=c,$ then write $P_{1}\left(  \frac
{c\mu+j-\left\{  R_{1}\right\}  }{ck}\right)  $ in place of $B_{1}\left(
\frac{c\mu+j-\left\{  R_{1}\right\}  }{ck}\right)  ,$ later add and subtract
the term $j=c$. Then
\begin{align*}
T_{9}  &  =2\pi i\sum\limits_{j=1}^{c}\sum\limits_{\mu=0}^{k-1}\sum
\limits_{v=0}^{k-1}f^{\ast}(b(\mu c+j+\left[  R_{1}\right]  ))f(c\left(
\left[  R_{2}+dj/c\right]  -v\right)  ) P_{1}\left(  \frac{c\mu+j-\left\{
R_{1}\right\}  }{ck}\right)  P_{1}\left(  \frac{v+\left\{  \left(
dj+\rho\right)  /c\right\}  }{k}\right) \\
&  \quad+2\pi iP_{1}(R_{2},A_{c})f^{\ast}(bR_{1}).
\end{align*}
Therefore, for arbitrary real number $R_{1},$\ we have
\begin{align*}
T_{9}  &  =2\pi i\sum\limits_{j=1}^{c}\sum\limits_{\mu=0}^{k-1}f^{\ast}(b(\mu
c+j+\left[  R_{1}\right]  ))P_{1}\left(  \frac{c\mu+j-\left\{  R_{1}\right\}
}{ck}\right) \\
&  \qquad\times\sum\limits_{v=0}^{k-1}f(c\left(  \left[  R_{2}+d\left(
j-\left\{  R_{1}\right\}  \right)  /c\right]  -v\right)  )P_{1}\left(
\frac{v+\left\{  \left(  dj+\rho\right)  /c\right\}  }{k}\right)  +2\pi
iP_{1}(R_{2},A_{c})f^{\ast}(bR_{1})\lambda_{R_{1}}\\
&  =2\pi i\sum\limits_{n=1}^{ck}f^{\ast}(bn)P_{1}\left(  \frac{n-R_{1}}%
{ck}\right)  P_{1}\left(  \frac{d\left(  n-R_{1}\right)  }{c}+R_{2}%
,A_{c}\right)  +2\pi iP_{1}(R_{2},A_{c})f^{\ast}(bR_{1})\lambda_{R_{1}%
}\mathbf{.}%
\end{align*}

\begin{definition}
Let $c$ and $d$ be coprime integers with $d\equiv0\left(  \operatorname{mod}%
k\right)  $ and $c>0$. For $bc\equiv-1$ $\left(  \operatorname{mod}d\right)
$, the generalized periodic Dedekind sum $s\left(  d,c;A_{b},A_{c};x,y\right)
$ is defined by
\[
s\left(  d,c;A_{b},A_{c};x,y\right)  =\sum\limits_{n=1}^{ck}f(bn)P_{1}\left(
\frac{n+y}{ck}\right)  P_{1}\left(  \frac{d\left(  n+y\right)  }{c}%
+x,A_{c}\right)  \text{.}%
\]

\end{definition}

Note that $s\left(  d,c;A_{b},A_{c};0,0\right)  =s\left(  d,c;A_{b}%
,A_{c}\right)  .$ Moreover $s\left(  d,c;A_{b},A_{c};x,y\right)  $ is the
natural generalization of the generalized Dedekind sum $s\left(
d,c;x,y\right)  $ \cite{13,12} and of the generalized Dedekind character sum
$s\left(  d,c;\chi,\overline{\chi};x,y\right)  $ \cite{2}.

So, we have obtained that
\begin{align}
T_{7}+T_{8}+T_{9} &  =2\pi is\left(  d,c;B_{b},A_{c},R_{2},-R_{1}\right)
-\frac{2\pi i}{c\left(  cz+d\right)  }B_{0}(B)P_{2}\left(  cR_{2}%
-dR_{1},A\right) \nonumber\\
&  \quad-\frac{2\pi i}{c}(cz+d)B_{0}(A)P_{2}(R_{1},B_{b})+2\pi iP_{1}%
(R_{2},A_{c})f^{\ast}(bR_{1})\lambda_{R_{1}}.\label{37}%
\end{align}

On the other hand, (\ref{L0}) gives%
\[
\lim_{s\rightarrow0}L(s;A_{c};-R_{2})=P_{1}(R_{2},A_{c})
\]
and then%
\begin{equation}
-2i\lambda_{R_{1}}f^{\ast}(bR_{1})\lim_{s\rightarrow0}\Gamma(s)\sin(\pi
s)L(s;A_{c};-R_{2})=-2i\pi\lambda_{R_{1}}f^{\ast}(bR_{1})P_{1}(R_{2}%
,A_{c}).\label{38}%
\end{equation}
Next, combining (\ref{36}), (\ref{37}) and (\ref{38}) yields%
\begin{align}
&  \lim_{s\rightarrow0}\Gamma(s)\left(  \frac{1}{\left(  cz+d\right)  ^{s}%
}G(Vz,s;A,B;r_{1},r_{2})-G(z,s;B_{-b},A_{-c};R_{1},R_{2})\right) \nonumber\\
&  =2\pi is\left(  d,c;B_{b},A_{c},R_{2},-R_{1}\right)  -\frac{2\pi
i}{c\left(  cz+d\right)  }B_{0}(B)P_{2}\left(  cR_{2}-dR_{1},A\right)
-\frac{2\pi i}{c}(cz+d)B_{0}(A)P_{2}(R_{1},B_{b}).\label{43}%
\end{align}
Now, let us consider the case $b\equiv c\equiv0\left(  \operatorname{mod}%
k\right)  .$ Then, from (\ref{40}), (\ref{13}) can be written as%
\begin{align}
&  \lim_{s\rightarrow0}\Gamma(s)\left(  \frac{1}{\left(  cz+d\right)  ^{s}%
}G(Vz,s;A,B;r_{1},r_{2})-G(z,s;A_{d},B_{a};R_{1},R_{2})\right) \nonumber\\
&  =-2if(-dR_{1})\lambda_{R_{1}}\lim_{s\rightarrow0}\Gamma(s)\sin(\pi
s)L(s;B_{-a};-R_{2})\nonumber\\
&  \quad+\sum\limits_{j=1}^{c}\sum\limits_{\mu=0}^{k-1}\sum\limits_{v=0}%
^{k-1}f^{\ast}(-a(\left[  R_{2}+d(j-\left\{  R_{1}\right\}  )/c\right]
-v+d\mu))f(-d(j+\left[  R_{1}\right]  )) I(z,0,c,d,R_{1},R_{2}).\label{39}%
\end{align}
Similar to (\ref{38}), we have%
\begin{equation}
2i\lambda_{R_{1}}f(-dR_{1})\lim_{s\rightarrow0}\Gamma(s)\sin(\pi
s)L(s;B_{-a};-R_{2})=-2\pi i\lambda_{R_{1}}f(-dR_{1})P_{1}(R_{2}%
,B_{-a}).\label{42}%
\end{equation}
Let%
\begin{align*}
&  \sum\limits_{j=1}^{c}\sum\limits_{\mu=0}^{k-1}\sum\limits_{v=0}%
^{k-1}f^{\ast}(-a(\left[  R_{2}+d(j-\left\{  R_{1}\right\}  )/c\right]
-v+d\mu))f(-d(j+\left[  R_{1}\right]  ))\\
&  \quad\times\left(  -\frac{\pi i}{cz+d}B_{2}\left(  \frac{v+\left\{  \left(
dj+\rho\right)  /c\right\}  }{k}\right)  -\pi i(cz+d)B_{2}\left(  \frac
{c\mu+j-\left\{  R_{1}\right\}  }{ck}\right)  \right. \\
&  \qquad\ \left.  +2\pi iB_{1}\left(  \frac{c\mu+j-\left\{  R_{1}\right\}
}{ck}\right)  B_{1}\left(  \frac{v+\left\{  \left(  dj+\rho\right)
/c\right\}  }{k}\right)  \right) \\
&  \ =T_{10}+T_{11}+T_{12}.
\end{align*}
To compute $T_{10}$, we first evaluate the sums over $\mu$ and over $v,$
respectively, and then use the facts $B_{2}\left(  \left\{  x\right\}
\right)  =2P_{2}\left(  x\right)  $ and $\rho=c\left\{  R_{2}\right\}
-d\left\{  R_{1}\right\}  ,$ to find that%
\begin{align*}
T_{10}  &  =-\frac{\pi i}{cz+d}\sum\limits_{j=1}^{c}\sum\limits_{v=0}%
^{k-1}f(-d(j+\left[  R_{1}\right]  ))B_{2}\left(  \frac{v+\left\{  \left(
dj+\rho\right)  /c\right\}  }{k}\right) \\
&  \quad\times\sum\limits_{\mu=0}^{k-1}f^{\ast}(-a(\left[  R_{2}+d(j-\left\{
R_{1}\right\}  )/c\right]  -v+d\mu))\\
&  =-\frac{2\pi i}{cz+d}B_{0}\left(  B\right)  \sum\limits_{j=1}%
^{c}f(-dj)P_{2}\left(  \frac{d\left(  j-R_{1}\right)  }{c}+\left\{
R_{2}\right\}  \right)  .
\end{align*}
Now, setting $c=qk$ (since $c\equiv0\left(  \operatorname{mod}k\right)  $),
then taking\ $j\rightarrow\mu k+r,$ $0\leq\mu\leq q-1,$ $1\leq r\leq k,$ gives%
\begin{align*}
T_{10}  &  =-\frac{2\pi i}{cz+d}B_{0}\left(  B\right)  \sum\limits_{\mu
=0}^{q-1}\sum\limits_{r=1}^{k}f(-dr)P_{2}\left(  \frac{d\mu}{q}+\frac{d\left(
r-R_{1}\right)  }{qk}+R_{2}\right) \\
&  =-\frac{2\pi i}{cz+d}\frac{B_{0}\left(  B\right)  }{q}\sum\limits_{r=1}%
^{k}f(-dr)P_{2}\left(  \frac{dr-dR_{1}+cR_{2}}{k}\right) \\
&  =-\frac{2\pi i}{c\left(  cz+d\right)  }B_{0}\left(  B\right)  P_{2}\left(
cR_{2}-dR_{1},A\right)  .
\end{align*}
Secondly, for computing $T_{11},$ evaluate the sums over $v$ and over $\mu,$
respectively, to find
\begin{align*}
T_{11}  &  =-\pi i(cz+d)\sum\limits_{j=1}^{c}\sum\limits_{\mu=0}^{k-1}f(-d(j+
\left[  R_{1}\right]  ))B_{2}\left(  \frac{c\mu+j-\left\{  R_{1}\right\}
}{ck}\right) \\
&  \quad\times\sum\limits_{v=0}^{k-1}f^{\ast}(-a(\left[  R_{2}+d(j-\left\{
R_{1}\right\}  )/c\right]  -v+d\mu))\\
&  =-\frac{2\pi i}{c}(cz+d)B_{0}\left(  B\right)  P_{2}(-R_{1},A_{d}).
\end{align*}
After similar arguments in evaluating $T_{9},$ we deduce for arbitrary real
number $R_{1}$ that \
\begin{align*}
T_{12}  &  =2\pi i\sum\limits_{j=1}^{c}\sum\limits_{\mu=0}^{k-1}%
\sum\limits_{v=0}^{k-1}f^{\ast}(-a(\left[  R_{2}+d(j-\left\{  R_{1}\right\}
)/c\right]  -v+d\mu))f(-d(j+\left[  R_{1}\right]  ))\\
&  \quad\times P_{1}\left(  \frac{c\mu+j-\left\{  R_{1}\right\}  }{ck}\right)
P_{1}\left(  \frac{v+\left\{  \left(  dj+\rho\right)  /c\right\}  }{k}\right)
+2\pi iP_{1}(R_{2},B_{-a})f(-dR_{1})\lambda_{R_{1}}\\
&  =2\pi i\sum\limits_{n=1}^{ck}f(-d(n+\left[  R_{1}\right]  ))P_{1}\left(
\frac{n-\left\{  R_{1}\right\}  }{ck}\right)  \sum\limits_{v=0}^{k-1}f^{\ast
}(av)P_{1}\left(  \frac{v+d\frac{n-\left\{  R_{1}\right\}  }{c}+R_{2}}%
{k}\right) \\
&\quad +2\pi iP_{1}(R_{2},B_{-a})f(-dR_{1})\lambda_{R_{1}}\\
&  =2\pi i\sum\limits_{n=1}^{ck}f(-dn)P_{1}\left(  \frac{n-R_{1}}{ck}\right)
P_{1}\left(  \frac{d\left(  n-R_{1}\right)  }{c}+R_{2},B_{-a}\right) +2\pi
iP_{1}(R_{2},B_{-a})f(-dR_{1})\lambda_{R_{1}}\text{\textbf{.}}%
\end{align*}

\begin{definition}
Let $c$ and $d$ be coprime integers with $c\equiv0\left(  \operatorname{mod}%
k\right)  $ and $c>0$. For $ad\equiv1$ $\left(  \operatorname{mod}c\right)  $,
the generalized periodic Dedekind sum $s\left(  d,c;A_{d},B_{a};x,y\right)  $
is defined by
\[
s\left(  d,c;A_{d},B_{a};x,y\right)  =\sum\limits_{n=1}^{ck}f(dn)P_{1}\left(
\frac{n+y}{ck}\right)  P_{1}\left(  \frac{d\left(  n+y\right)  }{c}%
+x,B_{a}\right)  .
\]

\end{definition}

So, we have obtained%
\begin{align}
T_{10}+T_{11}+T_{12}  &  =2\pi is\left(  d,c;A_{-d},B_{-a};R_{2}%
,-R_{1}\right)  -\frac{2\pi i}{c\left(  cz+d\right)  }B_{0}\left(  B\right)
P_{2}\left(  cR_{2}-dR_{1},A\right) \nonumber\\
&  \quad-\frac{2\pi i}{c}(cz+d)B_{0}\left(  B\right)  P_{2}(-R_{1},A_{d})+2\pi
iP_{1}(R_{2},B_{-a})f(-dR_{1})\lambda_{R_{1}}\mathbf{.}\label{41}%
\end{align}

Combining (\ref{39}), (\ref{42}) and (\ref{41}) yields%
\begin{align}
&  \lim_{s\rightarrow0}\Gamma(s)\left(  \frac{1}{\left(  cz+d\right)  ^{s}%
}G(Vz,s;A,B;r_{1},r_{2})-G(z,s;A_{d},B_{a};R_{1},R_{2})\right) \nonumber\\
&  =2\pi is\left(  d,c;A_{-d},B_{-a};R_{2},-R_{1}\right)  -\frac{2\pi
i}{c\left(  cz+d\right)  }B_{0}\left(  B\right)  P_{2}\left(  cR_{2}%
-dR_{1},A\right) \nonumber\\
&  \quad-\frac{2\pi i}{c}(cz+d)B_{0}\left(  B\right)  P_{2}(-R_{1}%
,A_{d}).\label{44}%
\end{align}

We now derive a reciprocity formula for $s\left(  d,c;B_{b},A_{c};R_{2}%
,R_{1}\right)  .$

\begin{theorem}
\label{rep2}Let $c$ and $d$ be coprime positive integers with $d\equiv0\left(
\operatorname{mod}k\right)  $. For $bc\equiv-1$ $\left(  \operatorname{mod}%
d\right) $,
\begin{align*}
&  s\left(  -c,d;A_{c},B_{-b};-R_{1},-R_{2}\right)  -s\left(  d,c;B_{b}%
,A_{c};R_{2},-R_{1}\right) \\
&  =P_{1}\left(  R_{2},B_{-b}\right)  P_{1}\left(  -R_{1},A_{-c}\right)
-\frac{1}{cd}B_{0}(B)P_{2}(cR_{2}-dR_{1},A)-\frac{d}{c}B_{0}(A)P_{2}\left(
R_{2},B_{-b}\right)  -\frac{c}{d}B_{0}\left(  B\right)  P_{2}\left(
R_{1},A_{c}\right)  .
\end{align*}

\end{theorem}

\begin{proof}
Replacing $z$ by $-1/z$ in (\ref{43}),\textbf{\ }%
\begin{align}
&  \lim_{s\rightarrow0}z^{s}\Gamma(s)\left(  \frac{1}{\left(  dz-c\right)
^{s}}G\left(  Tz,s;A,B;r_{1},r_{2}\right)  -z^{-s}G\left(  -1/z,s;B_{-b}%
,A_{-c};R_{1},R_{2}\right)  \right) \nonumber\\
&  =2\pi is\left(  d,c;B_{b},A_{c};R_{2},-R_{1}\right)  -\frac{2\pi i}{c}%
\frac{z}{\left(  dz-c\right)  }B_{0}(B)P_{2}(cR_{2}-dR_{1},A)-\frac{2\pi
i}{cz}(dz-c)B_{0}(A)P_{2}(R_{1},B_{b}).\label{45}%
\end{align}
Replacing $Vz$ by $Tz$ in (\ref{44}),
\begin{align}
&  \lim_{s\rightarrow0}\Gamma(s)\left(  \frac{1}{\left(  dz-c\right)  ^{s}%
}G(Tz,s;A,B;r_{1},r_{2})-G(z,s;A_{-c},B_{b};R_{2},-R_{1})\right) \nonumber\\
&  =2\pi is\left(  -c,d;A_{c},B_{-b};-R_{1},-R_{2}\right)  -\frac{2\pi i}%
{d}\frac{1}{dz-c}B_{0}(B)P_{2}(cR_{2}-dR_{1},A)-\frac{2\pi i}{d}%
(dz-c)B_{0}(B)P_{2}(-R_{2},A_{-c}).\label{46}%
\end{align}
Computing the equation (\ref{14}) for $\beta=b,$ $\alpha=c$ and the limiting
case $s=0$ gives%
\begin{align}
&  \lim_{s\rightarrow0}\Gamma(s)\left(  \frac{1}{z^{s}}G(-1/z,s;B_{-b}%
,A_{-c},R_{1},R_{2})-G(z,s;A_{-c},B_{b};R_{2},-R_{1})\right) \nonumber\\
&  =-2if(cR_{1})\lambda_{R_{1}}\lim_{s\rightarrow0}\Gamma(s)\sin(\pi
s)L(s;B_{-b};-R_{2})\nonumber\\
&  \quad+\sum\limits_{\mu=0}^{k-1}\sum\limits_{v=0}^{k-1}f(c(\mu+1+\left[
R_{1}\right]  ))f^{\ast}(-b\left(  \left[  R_{2}\right]  -v\right)
)I(z,0,1,0,R_{1},R_{2})\label{49}%
\end{align}
where%
\begin{align*}
&I(z,0,1,0,R_{1},R_{2}) \\
&=-\frac{\pi i}{z}B_{2}\left(  \frac{v+\left\{
R_{2}\right\}  }{k}\right)  -\pi izB_{2}\left(  \frac{\mu+1-\left\{
R_{1}\right\}  }{k}\right)  +2\pi iB_{1}\left(  \frac{\mu+1-\left\{
R_{1}\right\}  }{k}\right)  B_{1}\left(  \frac{v+\left\{  R_{2}\right\}  }%
{k}\right)  .
\end{align*}
As in before, we have\textbf{\ }%
\begin{align*}
&  \sum\limits_{v=0}^{k-1}\sum\limits_{\mu=0}^{k-1}f(c(\mu+1+\left[
R_{1}\right]  ))f^{\ast}(-b\left(  \left[  R_{2}\right]  -v\right)
)B_{2}\left(  \frac{v+\left\{  R_{2}\right\}  }{k}\right)  =2B_{0}%
(A)P_{2}\left(  R_{2},B_{-b}\right)  ,\\
&  \sum\limits_{\mu=0}^{k-1}\sum\limits_{v=0}^{k-1}f^{\ast}(-b\left(  \left[
R_{2}\right]  -v\right)  )f(c(\mu+1+\left[  R_{1}\right]  ))B_{2}\left(
\frac{\mu+1-\left\{  R_{1}\right\}  }{k}\right)  =2B_{0}\left(  B\right)
P_{2}\left(  R_{1},A_{c}\right)
\end{align*}
and%
\begin{align*}
&  \sum\limits_{v=0}^{k-1}f^{\ast}(-b\left(  \left[  R_{2}\right]  -v\right)
)B_{1}\left(  \frac{v+\left\{  R_{2}\right\}  }{k}\right)  \sum\limits_{\mu
=0}^{k-1}f(c(\mu+1+\left[  R_{1}\right]  ))B_{1}\left(  \frac{\mu+1-\left\{
R_{1}\right\}  }{k}\right) \\
&  =P_{1}\left(  R_{2},B_{-b}\right)  P_{1}\left(  -R_{1},A_{-c}\right)
+f(cR_{1})\lambda_{R_{1}}.
\end{align*}
From (\ref{38}) we have%
\[
-2if(cR_{1})\lambda_{R_{1}}\lim_{s\rightarrow0}\Gamma(s)\sin(\pi
s)L(s;B_{-b};-R_{2})=-2\pi if(cR_{1})\lambda_{R_{1}}P_{1}\left(  R_{2}%
,B_{-b}\right)  .
\]
Thus, combining above in (\ref{49}) gives
\begin{align}
&  \lim_{s\rightarrow0}\Gamma(s)\left(  \frac{1}{z^{s}}G(-1/z,s;B_{-b}%
,A_{-c},R_{1},R_{2})-G(z,s;A_{-c},B_{b};R_{2},-R_{1})\right) \nonumber\\
&  =2\pi iP_{1}\left(  R_{2},B_{-b}\right)  P_{1}\left(  -R_{1},A_{-c}\right)
-\frac{2\pi i}{z}B_{0}(A)P_{2}\left(  R_{2},B_{-b}\right)  -2\pi
izB_{0}\left(  B\right)  P_{2}\left(  R_{1},A_{c}\right)  .\label{48}%
\end{align}
Taking into account that%
\begin{align*}
&  \lim_{s\rightarrow0}z^{s}\Gamma(s)\left(  \frac{1}{\left(  dz-c\right)
^{s}}G\left(  Tz,s;A,B;r_{1},r_{2}\right)  -G(z,s;A_{-c},B_{b};R_{2}%
,-R_{1})\right) \\
&  =\lim_{s\rightarrow0}z^{s}\Gamma(s)\left(  \frac{1}{\left(  dz-c\right)
^{s}}G(V\left(  -1/z\right)  ,s;A,B;r_{1},r_{2}) -\frac{1}{z^{s}}G\left(
-1/z,s;B_{-b},A_{-c};R_{1},R_{2}\right)  \right) \\
&  \quad+\lim_{s\rightarrow0}z^{s}\Gamma(s)\left(  \frac{1}{z^{s}}G\left(
-1/z,s;B_{-b},A_{-c};R_{1},R_{2}\right)  -G(z,s;A_{-c},B_{b};R_{2}%
,-R_{1})\right)  ,
\end{align*}
(\ref{45}), (\ref{46}) and (\ref{48}) entail the reciprocity formula by
setting $z=c/d.$
\end{proof}

\section{Some series relations}

In this section we illustrate some transformation formulas of $A\left(
z,s;A_{\alpha},B_{\beta}\right)  :=A\left(  z,s;A_{\alpha},B_{\beta
};0,0\right)  $ for the special values of $A=\left\{  f(n)\right\}  $ and
$B=\left\{  f^{\ast}(n)\right\}  $\ by using (\ref{11}) and (\ref{13-a}) or
(\ref{17}). As applications of these formulas we present several relations
between various infinite series. For the rest of this paper we assume that
$a\equiv d\equiv0\left(  \operatorname{mod}k\right)  $ with $ad-bc=1$\ and
$\chi_{0}$ will denote the principle character of modulus $k\geq2$, i.e.,
\[
\chi_{0}\left(  n\right)  =%
\begin{cases}
1, & \text{if }\left(  n,k\right)  =1\\
0, & \text{if }\left(  n,k\right)  >1.
\end{cases}
\]
We first write $\Gamma(s)G(z,s;A,B)$ and $\Gamma(s)G(z,s;B_{-b},A_{-c})$ in
terms of $A\left(  z,s;A_{\alpha},B_{\beta}\right)  .$ Put $L(s;A_{c}%
)=L(s;A_{\alpha};0).$ From (\ref{11}), for $\alpha=\beta=1$ and $r_{1}%
=r_{2}=0$,
\begin{align}
\Gamma(s)G(z,s;A,B) &  =\left(  -2\pi i/k\right)  ^{s}k\left(  A\left(
z,s;A,\widehat{B}_{-1}\right)  +e(s/2)A\left(  z,s;A_{-1},\widehat{B}\right)
\right) \nonumber\\
&  \quad+\Gamma(s)f(0)\left(  L(s;B)+e(s/2)L(s;B_{-1})\right) \label{s1}%
\end{align}
and for $\alpha=-c,$ $\beta=-b$ and $r_{1}=r_{2}=0$,%
\begin{align}
\Gamma(s)G(z,s;B_{-b},A_{-c}) &  =\left(  -2\pi i/k\right)  ^{s}k\left(
A\left(  z,s;B_{-b},\widehat{A}_{c^{-1}}\right)  +e(s/2)A\left(
z,s;B_{b},\widehat{A}_{-c^{-1}}\right)  \right) \nonumber\\
&  \quad+\Gamma(s)f^{\ast}(0)\left(  L(s;A_{-c})+e(s/2)L(s;A_{c})\right)
.\label{s2}%
\end{align}

\textbf{1. }Let $A=\chi_{1}=\left\{  \chi_{1}(n)\right\}  $ and $B=\chi
_{2}=\left\{  \chi_{2}\left(  n\right)  \right\}  ,$ where $\chi_{1}$ and
$\chi_{2}$ are Dirichlet characters of modulus $k$.\textbf{\ }It is clear from
(\ref{47}) that
\begin{equation}
\widehat{f}(n)=\dfrac{1}{k}\sum\limits_{v=0}^{k-1}\chi(v)e^{-2\pi
inv/k}=\dfrac{1}{k}G(-n,\chi),\label{56}%
\end{equation}
where $G(n,\chi)$ is the Gauss sum.\textbf{\ }Then
\[
\widehat{A}=\left\{  \chi_{1}\left(  -1\right)  G(n,\chi_{1})/k\right\}
=\chi_{1}\left(  -1\right)  G_{1}/k\text{ and }\widehat{B}=\left\{  \chi
_{2}\left(  -1\right)  G(n,\chi_{2})/k\right\}  =\chi_{2}\left(  -1\right)
G_{2}/k.
\]
Thus, we deduce from (\ref{17}) after using (\ref{s1}) and (\ref{s2}) and
simplifying that
\begin{align}
&  \left(  A\left(  Vz,0;\chi_{1},G_{2}\right)  \frac{{}}{{}}-\chi_{1}\left(
c\right)  \chi_{2}\left(  b\right)  A\left(  z,0;\chi_{2},G_{1}\right)
\right)  \left(  1+\chi_{1}\left(  -1\right)  \chi_{2}\left(  -1\right)
\right) \nonumber\\
&  =2\pi i\chi_{1}\left(  c\right)  \chi_{2}\left(  b\right)  s\left(
d,c;\chi_{2},\chi_{1}\right)  -\frac{\pi i}{c(cz+d)}B_{0}\left(  \chi
_{2}\right)  B_{2}\left(  \chi_{1}\right)  -2\pi i\frac{\chi_{2}\left(
b\right)  }{c}\left(  cz+d\right)  B_{0}\left(  \chi_{1}\right)  P_{2}\left(
0,\chi_{2}\right)  .\label{s3}%
\end{align}
Notice that $P_{m}\left(  x,\chi\right)  $ and $B_{m}\left(  \chi\right)
$\ correspond to the generalized Bernoulli functions and numbers denoted by
$\chi\left(  -1\right)  B_{m}\left(  x,\overline{\chi}\right)  $ and
$\chi\left(  -1\right)  B_{m}(\overline{\chi}),$ respectively, given in
\cite{3}. In the sequel we need the relations \textbf{\ }%
\begin{equation}
B_{2m+1}\left(  \chi\right)  =0\text{ for even }\chi,\text{ }B_{2m}\left(
\chi\right)  =0\text{ for odd }\chi,\text{ and }B_{0}(\chi)=0\text{ for }%
\chi\not =\chi_{0}.\label{58}%
\end{equation}

If $\chi_{1}$ and $\chi_{2}$ are non-principle Dirichlet characters with
$\chi_{1}\left(  -1\right)  \chi_{2}\left(  -1\right)  =1,$ then, by
(\ref{58}), (\ref{s3}) reduces to
\begin{equation}
A\left(  Vz,0;\chi_{1},G_{2}\right)  -\chi_{1}\left(  c\right)  \chi
_{2}\left(  b\right)  A\left(  z,0;\chi_{2},G_{1}\right)  =\pi i\chi
_{1}\left(  c\right)  \chi_{2}\left(  b\right)  s\left(  d,c;\chi_{2},\chi
_{1}\right)  .\nonumber
\end{equation}

For a principle character $\chi_{0}$ the Gauss sum reduces to the Ramanujan
sum
\[
c_{k}(n):=\sum\limits_{\substack{v=1 \\\left(  v,k\right)  =1}}^{k}e^{2\pi
inv/k}.
\]
\textbf{\ }Supposing $\chi_{1}\mathbf{=}\chi_{2}=\chi_{0},$ (\ref{s3}) takes
the form
\[
A\left(  Vz,0;\chi_{0},c_{k}\right)  -A\left(  z,0;\chi_{0},c_{k}\right)  =\pi
is\left(  d,c;\chi_{0},\chi_{0}\right)  -\frac{\pi i}{2ck}\left(  \frac
{1}{cz+d}+cz+d\right)  \phi\left(  k\right)  B_{2}\left(  \chi_{0}\right)  ,
\]
where we have used that $B_{0}\left(  \chi_{0}\right)  =\phi\left(  k\right)
/k$ by (\ref{10}). Here $\phi\left(  k\right)  $ stands for the Euler $phi$
function, i.e., the number of positive integers not exceeding $k$ which are
relatively prime to $k.$

Now we utilize (\ref{13-a}) instead of (\ref{17}), with the use of
(\ref{int1}) for $s=-2N,$ where $N$ is an integer. Using (\ref{s1}) and
(\ref{s2}) in (\ref{13-a}), with $Vz=-1/z,$ and assuming that $\chi_{1}\left(
-1\right)  \chi_{2}\left(  -1\right)  =1,$ we find that%
\begin{align}
&  z^{2N}A\left(  -1/z,-2N;\chi_{1},G_{2}\right)  -\chi_{1}\left(  -1\right)
A\left(  z,-2N;\chi_{2},G_{1}\right) \nonumber\\
&  =\frac{\left(  2\pi i\right)  ^{2N+1}}{2}\sum\limits_{m=0}^{2N+2}%
\sum\limits_{\mu=0}^{k-1}\sum\limits_{v=0}^{k-1}\frac{\chi_{1}(-v)\chi
_{2}(-\mu-1)}{m!\left(  2N+2-m\right)  !}B_{2N+2-m}\left(  \frac{v}{k}\right)
B_{m}\left(  \frac{\mu+1}{k}\right)  \left(  -z\right)  ^{m-1}\nonumber\\
&  =-\frac{\left(  2\pi i\right)  ^{2N+1}}{2k^{2N}}\sum\limits_{m=0}%
^{2N+2}\left(  -1\right)  ^{m}\frac{B_{m}\left(  \chi_{2}\right)  }{m!}%
\frac{B_{2N+2-m}\left(  \chi_{1}\right)  }{\left(  2N+2-m\right)  !}%
z^{m-1},\label{s4}%
\end{align}
which can be equally written as
\begin{align}
&  z^{2N}\sum\limits_{n=1}^{\infty}\frac{G\left(  n,\chi_{2}\right)  G\left(
-n/z,\chi_{1}\right)  }{1-e^{-2\pi in/z}}n^{-2N-1}-\chi_{1}\left(  -1\right)
\sum\limits_{n=1}^{\infty}\frac{G\left(  n,\chi_{1}\right)  G\left(
nz,\chi_{2}\right)  }{1-e^{2\pi inz}}n^{-2N-1}\nonumber\\
&  =-\frac{\left(  2\pi i\right)  ^{2N+1}}{2k^{2N}}\sum\limits_{m=0}%
^{2N+2}\left(  -1\right)  ^{m}\frac{B_{m}\left(  \chi_{2}\right)  }{m!}%
\frac{B_{2N+2-m}\left(  \chi_{1}\right)  }{\left(  2N+2-m\right)  !}%
z^{m-1}.\label{s4-1}%
\end{align}
Indeed, setting $m=v+hk,$ $0\leq v\leq k-1,$ $0\leq h<\infty,$ in the double
sum%
\[
A\left(  z,-2N;\chi_{1},G_{2}\right)  =\sum\limits_{m=1}^{\infty}%
\sum\limits_{n=1}^{\infty}\chi_{1}\left(  m\right)  G\left(  n,\chi
_{2}\right)  e^{2\pi inmz/k}n^{-2N-1}%
\]
we have%
\begin{align}
A\left(  z,-2N;\chi_{1},G_{2}\right)   &  =\sum\limits_{n=1}^{\infty}G\left(
n,\chi_{2}\right)  n^{-2N-1}\sum\limits_{v=0}^{k-1}\chi_{1}\left(  v\right)
e^{2\pi ivnz/k}\sum\limits_{h=0}^{\infty}e^{2\pi inzh}\nonumber\\
&  =\sum\limits_{n=1}^{\infty}\frac{G\left(  n,\chi_{2}\right)  G\left(
nz,\chi_{1}\right)  }{1-e^{2\pi inz}}n^{-2N-1}.\label{s4-a}%
\end{align}
Let $z=\pi i/k\gamma$ in (\ref{s4-1}), where $\gamma>0,$ and determine
$\theta>0$ by $\gamma\theta=\pi^{2}/k^{2}$. Then we see that
\begin{align}
&  \gamma^{-N}\sum\limits_{n=1}^{\infty}\frac{G\left(  n,\chi_{2}\right)
G\left(  ink\gamma/\pi,\chi_{1}\right)  }{1-e^{-2nk\gamma}}n^{-2N-1}%
\nonumber\\
&  \quad-\left(  -\theta\right)  ^{-N}\chi_{1}\left(  -1\right)
\sum\limits_{n=1}^{\infty}\frac{G\left(  n,\chi_{1}\right)  G\left(
ink\theta/\pi,\chi_{2}\right)  }{1-e^{-2nk\theta}}n^{-2N-1}\nonumber\\
&  =-k2^{2N}\sum\limits_{m=0}^{2N+2}\left(  -i\right)  ^{m}\frac{B_{m}\left(
\chi_{2}\right)  }{m!}\frac{B_{2N+2-m}\left(  \chi_{1}\right)  }{\left(
2N+2-m\right)  !}\gamma^{N+1-m/2}\theta^{m/2}.\label{s4-2}%
\end{align}
Setting $\chi_{1}=\chi_{2}=\chi,$ $\gamma=\theta=\pi/k$ in (\ref{s4-2})\ or
$z=i$ in (\ref{s4-1}) and assuming that $\chi\left(  -1\right)  \left(
-1\right)  ^{N}=-1$ give%
\begin{equation}
\sum\limits_{n=1}^{\infty}\frac{G\left(  n,\chi\right)  G\left(
in,\chi\right)  }{1-e^{-2\pi n}}n^{-2N-1}=-\frac{k\left(  2\pi/k\right)
^{2N+1}}{4}\sum\limits_{m=0}^{2N+2}\left(  -i\right)  ^{m}\frac{B_{m}\left(
\chi\right)  }{m!}\frac{B_{2N+2-m}\left(  \chi\right)  }{\left(
2N+2-m\right)  !}.\label{s5}%
\end{equation}
As a consequence of (\ref{s5}) we may write
\begin{align*}
\sum\limits_{n=1}^{\infty}\frac{G\left(  n,\chi\right)  G\left(
in,\chi\right)  }{1-e^{-2\pi n}}n^{-1}  &  =\frac{\pi i}{2}B_{1}\left(
\chi\right)  B_{1}\left(  \chi\right)  ,\text{ for odd }\chi,\\
\sum\limits_{n=1}^{\infty}\frac{G\left(  n,\chi\right)  G\left(
in,\chi\right)  }{1-e^{-2\pi n}}n  &  =-\frac{k^{2}}{8\pi}B_{0}\left(
\chi\right)  B_{0}\left(  \chi\right)  ,\text{ for even }\chi,\\
\sum\limits_{n=1}^{\infty}\frac{G\left(  n,\chi\right)  G\left(
in,\chi\right)  }{1-e^{-2\pi n}}n^{2M-1}  &  =0,\text{ for }\chi\left(
-1\right)  \left(  -1\right)  ^{M}=-1,\text{ }M\geq2.
\end{align*}

Now differentiate both sides of (\ref{s4}) with respect to $z$ and then put
$\chi_{1}=\chi_{2}=\chi$ and $z=i.$ Under the assumption $\chi\left(
-1\right)  \left(  -1\right)  ^{N}=1,$ we obtain%
\begin{align}
&  \sum\limits_{m=1}^{\infty}\sum\limits_{n=1}^{\infty}\chi\left(  m\right)
G\left(  n,\chi\right)  e^{-2\pi nm/k}n^{-2N-1}\left(  N+\frac{2\pi mn}%
{k}\right) \nonumber\\
&  =-\frac{k\left(  2\pi/k\right)  ^{2N+1}}{4}\sum\limits_{m=0}^{2N+2}\left(
-i\right)  ^{m}\left(  m-1\right)  \frac{B_{m}\left(  \chi_{2}\right)  }%
{m!}\frac{B_{2N+2-m}\left(  \chi_{1}\right)  }{\left(  2N+2-m\right)
!}.\label{s6}%
\end{align}

Now we examine (\ref{s4-2}) for $\chi_{1}=\chi_{2}=\chi.$

\begin{example}
Let, in (\ref{s4-2}), $\chi_{1}=\chi_{2}=\chi$ be the primitive character of
modulus $k=4$ defined by\
\begin{equation}
\chi\left(  n\right)  =\left\{
\begin{array}
[c]{cl}%
0, & n\text{ even,}\\
1, & n\equiv1\left(  \operatorname{mod}4\right)  ,\\
-1, & n\equiv3\left(  \operatorname{mod}4\right)  .
\end{array}
\right. \label{pk}%
\end{equation}
Since $B_{m}\left(  1-x\right)  =\left(  -1\right)  ^{m}B_{m}\left(  x\right)
$ and
\begin{equation}
B_{m}\left(  \frac{1}{4}\right)  =2^{-m}B_{m}\left(  \frac{1}{2}\right)
-m4^{-m}E_{m-1},m\geq1,\label{s6-ek}%
\end{equation}
where $E_{m}$ is the $m$--th Euler number \cite[p. 806]{ab}, we have, for odd
$m$, that%
\[
B_{m}\left(  \chi\right)  =4^{m-1}\sum\limits_{v=0}^{3}\chi\left(  -v\right)
B_{m}\left(  \frac{v}{4}\right)  =\frac{m}{2}E_{m-1}.
\]
Using $G\left(  ix,\chi\right)  =e^{-\pi x}\left(  e^{\pi x/2}-e^{-\pi
x/2}\right)  $ and $G\left(  n,\chi\right)  =\overline{\chi}\left(  n\right)
G\left(  1,\chi\right)  =2i\chi\left(  n\right)  ,$ and simplifying give, for
$\gamma\theta=\pi^{2}/16,$ that%
\begin{align}
&  \gamma^{-N}\sum\limits_{n=1}^{\infty}\chi\left(  n\right)  \frac
{\text{sech}\left(  2n\gamma\right)  }{n^{2N+1}}+\left(  -\theta\right)
^{-N}\sum\limits_{n=1}^{\infty}\chi\left(  n\right)  \frac{\text{sech}\left(
2n\theta\right)  }{n^{2N+1}}\nonumber\\
&  =2^{2N}\frac{\pi}{4}\sum\limits_{m=0}^{N}\left(  -1\right)  ^{m}%
\frac{E_{2m}}{\left(  2m\right)  !}\frac{E_{2N-2m}}{\left(  2N-2m\right)
!}\gamma^{N-m}\theta^{m},\label{s4-3}%
\end{align}
where we have used the fact that $G\left(  n,\chi\right)  =\overline{\chi
}\left(  n\right)  G\left(  1,\chi\right)  $ for primitive character $\chi$
\cite[p. 168]{1}.
\end{example}

We remark that in the case when $\chi$ is primitive (\ref{s4})--(\ref{s6})
have been given by Berndt \cite{7,8}. Formula (\ref{s4-3}) is also found in
Ramanujan's Notebooks; see Entry 21 (ii) on p. 276 of Berndt \cite{17}. Also
Entries 14, 15 and 25 (vii), (viii), (ix) in \cite{17} are special cases of
(\ref{s4-3}). Moreover, since formula (\ref{s4-1}) is a generalization of
Theorem 4.2 in \cite{8} from primitive characters to Dirichlet characters, the
results of Theorem 4.2 in \cite{8} are also consequences of (\ref{s4-1}).
Historical informations and literature of (\ref{s4-3}) and their results can
be found in \cite{8}.

\begin{example}
Let $\chi_{1}=\chi_{2}=\chi_{0}$ with modulus $k=4$ in (\ref{s4-2}). Then we
have
\[
G\left(  ix,\chi_{0}\right)  =e^{-\pi x}\left(  e^{\pi x/2}+e^{-\pi
x/2}\right)  ,\text{ \ }G\left(  n,\chi_{0}\right)  =2\left(  -1\right)
^{n}\cos\left(  \pi n/2\right)
\]
and by (\ref{s6-ek})
\[
B_{2m+1}\left(  \chi_{0}\right)  =0\text{ and }B_{2m}\left(  \chi_{0}\right)
=2^{2m-1}B_{2m}\left(  1/2\right)  ,\text{ }m\geq0\text{.}%
\]
Replacing $\gamma$ and $\theta$ by $\gamma/4$ and $\theta/4,$ respectively,
and simplifying give, for $\gamma\theta=\pi^{2},$ that%
\begin{align}
&  \gamma^{-N}\sum\limits_{n=1}^{\infty}\left(  -1\right)  ^{n}\frac
{\text{csch}\left(  n\gamma\right)  }{n^{2N+1}}-\left(  -\theta\right)
^{-N}\sum\limits_{n=1}^{\infty}\left(  -1\right)  ^{n}\frac{\text{csch}\left(
n\theta\right)  }{n^{2N+1}}\nonumber\\
&  =-2^{2N+1}\sum\limits_{m=0}^{N+1}\left(  -1\right)  ^{m}\frac{B_{2m}\left(
1/2\right)  }{\left(  2m\right)  !}\frac{B_{2N+2-2m}\left(  1/2\right)
}{\left(  2N+2-2m\right)  !}\gamma^{N+1-m}\theta^{m}.\label{s4-4}%
\end{align}
It follows for $N=2M+1$ that\textbf{\ }
\[
\sum\limits_{n=1}^{\infty}\left(  -1\right)  ^{n}\frac{\text{csch}\left(
n\pi\right)  }{n^{4M+3}}=-\left(  2\pi\right)  ^{4M+3}\sum\limits_{m=0}%
^{2M+2}\left(  -1\right)  ^{m}\frac{B_{2m}\left(  1/2\right)  }{\left(
2m\right)  !}\frac{B_{4M+4-2m}\left(  1/2\right)  }{\left(  4M+4-2m\right)
!},
\]
which was first proved by Cauchy as cited by Berndt \cite{22}.
\end{example}

(\ref{s4-4}) has been also established by Berndt \cite[Theorem 3.1]{22}.
Historical informations and literature of (\ref{s4-4}) and their results can
be found in \cite{22}.

\textbf{2. }Let $A=\left\{  f(n)\right\}  =\left\{  G(n,\chi_{1})\right\}
=G_{1} $ and $B=\left\{  f^{\ast}(n)\right\}  =\left\{  G(n,\chi_{2})\right\}
=G_{2},$ where $\chi_{1}$ and $\chi_{2}$ are Dirichlet characters of modulus
$k$. Then, $A_{\alpha}=\overline{\chi_{1}}\left(  \alpha\right)  G_{1}$ and
$B_{\alpha}=\overline{\chi_{2}}\left(  \alpha\right)  G_{2}$\ for $\left(
\alpha,k\right)  =1,$\ and from (\ref{47}) $\widehat{A}=\left\{  \chi
_{1}(n)\right\}  $ and $\widehat{B}=\left\{  \chi_{2}(n)\right\}  .$ So that
we find from (\ref{s1}), (\ref{s2}) and (\ref{17}) that
\begin{align}
&  k\left(  A\left(  Vz,0;G_{1},\chi_{2}\right)  \frac{{}}{{}}-\overline
{\chi_{1}}\left(  c\right)  \overline{\chi_{2}}\left(  b\right)  A\left(
z,0;G_{2},\chi_{1}\right)  \right)  \left(  \chi_{1}\left(  -1\right)
+\chi_{2}\left(  -1\right)  \right) \label{s7}\\
&  +\lim_{s\rightarrow0}\left(  \Gamma\left(  s\right)  \left(  cz+d\right)
^{-s}\left(  1+e\left(  -s/2\right)  \chi_{2}\left(  -1\right)  \right)
L\left(  s,G_{2}\right)  G(0,\chi_{1})\frac{{}}{{}}\right. \nonumber\\
&  \left.  \frac{{}}{{}}\qquad-\Gamma\left(  s\right)  \left(  1+e\left(
-s/2\right)  \chi_{1}\left(  -1\right)  \right)  L\left(  s,G_{1}\right)
G(0,\chi_{2})\overline{\chi_{1}}\left(  -c\right)  \right) \nonumber\\
&  =2\pi i\overline{\chi_{1}}\left(  c\right)  \overline{\chi_{2}}\left(
b\right)  s\left(  d,c;G_{2},G_{1}\right)  -\frac{\pi i}{c(cz+d)}B_{0}%
(G_{2})B_{2}(G_{1})-2\pi i\overline{\chi_{2}}\left(  b\right)  \frac{cz+d}%
{c}B_{0}(G_{1})P_{2}\left(  0,G_{2}\right)  .\nonumber
\end{align}
Suppose that $\chi_{1}\left(  -1\right)  \chi_{2}\left(  -1\right)
=1$\textbf{.} Thus (\ref{s7}) yields
\begin{align}
&  A\left(  Vz,0;G_{1},\chi_{2}\right)  -\overline{\chi_{1}}\left(  c\right)
\overline{\chi_{2}}\left(  b\right)  A\left(  z,0;G_{2},\chi_{1}\right)
\nonumber\\
&  \quad=\frac{\pi i}{k}\overline{\chi_{1}}\left(  -c\right)  \overline
{\chi_{2}}\left(  b\right)  s\left(  d,c;G_{2},G_{1}\right)  +\delta\left(
z;c,d;\chi_{1},\chi_{2}\right)  ,\label{s7-1}%
\end{align}
where
\begin{equation}
\delta\left(  z;c,d;\chi_{1},\chi_{2}\right)  =\left\{
\begin{array}
[c]{ll}%
\dfrac{c_{k}(0)}{k}P_{1}\left(  0,c_{k}\right)  \log\left(  cz+d\right)  , &
\text{if }\chi_{1}=\chi_{2}=\chi_{0},\\
0, & \text{if }\chi_{1}\not =\chi_{0}\text{ and }\chi_{2}\not =\chi
_{0}\mathbf{,}%
\end{array}
\right. \label{s7-2}%
\end{equation}
and we have used (\ref{L0}) and that $G(0,\chi)=0$ for $\chi\not =\chi_{0}$
and $B_{0}(G)=0.$

Note that the number $P_{1}\left(  0,G\right)  $ can be evaluated as%
\begin{align}
P_{1}\left(  0,G\right)   &  =\sum\limits_{v=0}^{k-1}G(-v,\chi)P_{1}\left(
\dfrac{v}{k}\right)  =\sum\limits_{j=1}^{k-1}\chi\left(  j\right)  \frac
{1}{e^{-2\pi ij/k}-1}\nonumber\\
&  =\frac{i}{2}\sum\limits_{j=1}^{k-1}\chi\left(  j\right)  \cot\left(  \pi
j/k\right)  -\frac{1}{2}\sum\limits_{j=1}^{k-1}\chi\left(  j\right)
.\label{54}%
\end{align}
In fact, the numbers $P_{r}\left(  0,G\right)  ,$ $r\geq1,$ are closely
related to the Dirichlet $L$--function $L\left(  r,\chi\right)  .$ Let $\chi$
be the Dirichlet character of modulus $k\geq2$ and put $\chi\left(  -1\right)
=\left(  -1\right)  ^{l}.$ Alkan \cite[Theorem 1]{4} shows that if $l$ and
$r\geq1$ have the same parity, then
\begin{equation}
2k\frac{\left(  -1\right)  ^{l+1}r!}{\left(  2\pi i\right)  ^{r}}L\left(
r,\chi\right)  =\sum\limits_{q=0}^{2\left[  \frac{r}{2}\right]  }\binom{r}%
{q}B_{q}S\left(  r-q,\chi\right)  ,\label{54-1}%
\end{equation}
where $S\left(  m,\chi\right)  =\sum\limits_{j=1}^{k}\left(  j/k\right)
^{m}G\left(  j,\chi\right)  $. Now, using the facts that
\[
\sum\limits_{q=0}^{r}\binom{r}{q}B_{q}x^{r-q}=B_{r}\left(  x\right)  \text{
and }B_{r}\left(  1-x\right)  =\left(  -1\right)  ^{r}B_{r}\left(  x\right)
\]
the right hand side of (\ref{54-1}) may be written
\[
\left(  -1\right)  ^{r}\sum\limits_{j=0}^{k-1}G(-j,\chi)B_{r}\left(  \dfrac
{j}{k}\right)  =\left(  -1\right)  ^{r}r!k^{1-r}P_{r}\left(  0,G\right)  ,
\]
which yields
\[
P_{r}\left(  0,G\right)  =-2\left(  \frac{k}{2\pi i}\right)  ^{r}L\left(
r,\chi\right)  \text{, if }l\text{ and }r\geq1\text{ have the same parity.}%
\]
Furthermore\textbf{\ }since $\cot x$ is an odd function, it is seen for even
Dirichlet character $\chi$ that\textbf{\ }$P_{1}\left(  0,G\right)
=0$\textbf{\ }if $\chi\not =\chi_{0},$ and $P_{1}\left(  0,G\right)
=P_{1}\left(  0,c_{k}\right)  =-\phi\left(  k\right)  /2$\ if $\chi=\chi_{0}$.
Therefore, for $Vz=-1/z,$ (\ref{s7-1}) may be written explicitly as
\begin{align*}
A\left(  -1/z,0;G_{1},\chi_{2}\right)  +A\left(  z,0;G_{2},\chi_{1}\right)
&  =-\frac{k}{\pi i}L\left(  1,\chi_{1}\right)  L\left(  1,\chi_{2}\right)
,\text{ for odd }\chi_{1}\text{ and }\chi_{2},\\
A\left(  -1/z,0;G_{1},\chi_{2}\right)  -A\left(  z,0;G_{2},\chi_{1}\right)
&  =0,\text{ for even }\chi_{1}\not =\chi_{0}\text{ and }\chi_{2}\not =%
\chi_{0},\\
A\left(  -1/z,0;c_{k},\chi_{0}\right)  -A\left(  z,0;c_{k},\chi_{0}\right)
&  =\frac{\phi\left(  k\right)  ^{2}}{2k}\left(  \frac{\pi i}{2}-\log
z\right)  ,\text{ for }\chi_{1}=\chi_{2}=\chi_{0}.
\end{align*}
The counterpart of (\ref{s4-a}) for $A\left(  z,0;G_{1},\chi_{2}\right)  $
is\textbf{\ }%
\[
A\left(  z,0;G_{1},\chi_{2}\right)  =\sum\limits_{v=1}^{k-1}\chi_{1}\left(
v\right)  \sum\limits_{n=1}^{\infty}\frac{\chi_{2}\left(  n\right)
}{1-e^{2\pi i\left(  v+nz\right)  /k}}n^{-1}.
\]
As before, let $z=\pi i/k\gamma$, where $\gamma>0,$ and determine $\theta>0$
by $\gamma\theta=\pi^{2}/k^{2}.$ Then, for $Vz=-1/z,$ (\ref{s7-1}) becomes
\begin{align}
&  \sum\limits_{n=1}^{\infty}\sum\limits_{v=0}^{k-1}\frac{\chi_{1}\left(
v\right)  \chi_{2}\left(  n\right)  }{1-e^{2\pi iv/k-2n\gamma}}n^{-1}-\chi
_{2}\left(  -1\right)  \sum\limits_{n=1}^{\infty}\sum\limits_{v=0}^{k-1}%
\frac{\chi_{2}\left(  v\right)  \chi_{1}\left(  n\right)  }{1-e^{2\pi
iv/k-2n\theta}}n^{-1}\nonumber\\
&  =\frac{\pi i}{k}\chi_{1}\left(  -1\right)  P_{1}\left(  0,G_{1}\right)
P_{1}\left(  0,G_{2}\right)  +\delta\left(  z;1,0;\chi_{1},\chi_{2}\right)
,\label{s9}%
\end{align}
where $\delta\left(  z;1,0;\chi_{1},\chi_{2}\right)  $ is given by (\ref{s7-2}).

\begin{example}
Let $\chi_{1}=\chi_{2}=\chi$ be the Dirichlet character of modulus $k=6$
defined by
\begin{equation}
\chi\left(  n\right)  =\left\{
\begin{array}
[c]{cl}%
1, & n\equiv1\left(  \operatorname{mod}6\right)  ,\\
-1, & n\equiv5\left(  \operatorname{mod}6\right)  ,\\
0, & \left(  n,6\right)  >1
\end{array}
\right. \label{s9-3}%
\end{equation}
in (\ref{s9}). A simple calculation gives
\[
\sum\limits_{n=1}^{\infty}\frac{\chi\left(  n\right)  n^{-1}}{2\cosh\left(
2n\gamma\right)  -1}+\sum\limits_{n=1}^{\infty}\frac{\chi\left(  n\right)
n^{-1}}{2\cosh\left(  2n\theta\right)  -1}=\frac{\pi}{2\sqrt{3}}\text{, for
}\gamma\theta=\pi^{2}/36
\]
and%
\[
\sum\limits_{n=1}^{\infty}\frac{\chi\left(  n\right)  n^{-1}}{2\cosh\left(
n\pi/3\right)  -1}=\frac{\pi}{4\sqrt{3}}\text{, for }\gamma=\theta=\pi/6.
\]

\end{example}

\begin{example}
Let $\chi_{1}=\chi_{2}=\chi$ be the primitive character of modulus $k=3$
defined by
\[
\chi\left(  n\right)  =\left\{
\begin{array}
[c]{cl}%
0, & n\equiv0\left(  \operatorname{mod}3\right)  ,\\
1, & n\equiv1\left(  \operatorname{mod}3\right)  ,\\
-1, & n\equiv2\left(  \operatorname{mod}3\right)
\end{array}
\right.
\]
(\ref{s9}). Then, for $\gamma\theta=\pi^{2}/9$
\[
\sum\limits_{n=1}^{\infty}\frac{\chi\left(  n\right)  n^{-1}}{2\cosh
2n\gamma+1}+\sum\limits_{n=1}^{\infty}\frac{\chi\left(  n\right)  n^{-1}%
}{2\cosh2n\theta+1}=\frac{\pi}{9\sqrt{3}}.
\]

\end{example}

\begin{example}
Let $\chi_{1}=\chi_{2}=\chi_{0}$ with modulus $k=4$ in (\ref{s9})$.$ Then we
have, for $\gamma\theta=\pi^{2},$ that%
\begin{equation}
\sum\limits_{n=0}^{\infty}\frac{1}{2n+1}\left(  \frac{1}{1+e^{-\left(
2n+1\right)  \gamma}}-\frac{1}{1+e^{-\left(  2n+1\right)  \theta}}\right)
=\frac{1}{8}\left(  \log\gamma-\log\theta\right)  ,\label{s9-1}%
\end{equation}
which is same with Corollary 4.3 in \cite{22}. Differentiating both sides of
(\ref{s9-1}) with respect to $\gamma$ gives
\begin{gather*}
\gamma\sum\limits_{n=0}^{\infty}\text{sech}^{2}\left(  \left(  2n+1\right)
\frac{\gamma}{2}\right)  +\theta\sum\limits_{n=0}^{\infty}\text{sech}%
^{2}\left(  \left(  2n+1\right)  \frac{\theta}{2}\right)  =1,\\
\sum\limits_{n=0}^{\infty}\text{sech}^{2}\left(  \left(  2n+1\right)
\frac{\pi}{2}\right)  =\frac{1}{2\pi},
\end{gather*}
i.e., Corollaries 4.4 and 4.5 in \cite{22}, respectively.
\end{example}

\textbf{3. }Let $A=\left\{  \left(  -1\right)  ^{n}\chi_{1}(n)\right\}  $ and
$B=\left\{  \left(  -1\right)  ^{n}\chi_{2}(n)\right\}  ,$ where $\chi_{1}$
and $\chi_{2}$ are Dirichlet characters of modulus $k.$ Assume that $k$ is
even and $\chi_{1}\left(  -1\right)  \chi_{2}\left(  -1\right)  =1.$ Then,
from (\ref{56}),
\begin{align*}
&  \widehat{A}=\frac{\chi_{1}\left(  -1\right)  }{k} \left\{  G(n+k/2,\chi
_{1})\right\}  =\frac{\chi_{1}\left(  -1\right)  }{k}G_{1}^{\ast},\\
&  \widehat{B}=\frac{\chi_{2}\left(  -1\right)  }{k}\left\{  G(n+k/2,\chi
_{2})\right\}  =\frac{\chi_{2}\left(  -1\right)  }{k}G_{2}^{\ast}.
\end{align*}
Since $ad-bc=1$ and $a\equiv d\equiv0\left(  \operatorname{mod}k\right)  ,$
$b$ and $c$ must be odd. It can be seen for $\left(  k,\alpha\right)  =1$ that
$G(\alpha n+k/2,\chi)=\overline{\chi}\left(  \alpha\right)  G(n+k/2,\chi),$
which implies $\widehat{A}_{\alpha}=\overline{\chi_{1}}\left(  -\alpha\right)
G_{1}^{\ast}/k$ and $\widehat{B}_{\alpha}=\overline{\chi_{2}}\left(
-\alpha\right)  G_{2}^{\ast}/k.$ Also
\begin{equation}
P_{m}\left(  x,A_{c}\right)  =\chi_{1}\left(  -c\right)  k^{m-1}%
\sum\limits_{v=0}^{k-1}\left(  -1\right)  ^{v}\chi_{1}\left(  v\right)
P_{m}\left(  \dfrac{v+x}{k}\right)  =\chi_{1}\left(  -c\right)  P_{m}^{\ast
}\left(  x,\overline{\chi_{1}}\right) \label{s8-3}%
\end{equation}
and
\begin{align}
s\left(  d,c;B_{b},A_{c}\right)   &  =\chi_{1}\left(  -c\right)  \chi
_{2}\left(  b\right)  \sum\limits_{n=1}^{ck}\left(  -1\right)  ^{n}\chi
_{2}\left(  n\right)  P_{1}\left(  \dfrac{n}{ck}\right)  P_{1}^{\ast}\left(
\frac{dn}{c},\overline{\chi_{1}}\right) \label{s8-4}\\
&  =\chi_{1}\left(  -c\right)  \chi_{2}\left(  b\right)  s^{\ast}\left(
d,c;\chi_{2},\chi_{1}\right)  ,\nonumber
\end{align}
where $P_{m}^{\ast}\left(  x,\overline{\chi_{1}}\right)  $ and $s^{\ast
}\left(  d,c;\chi_{2},\chi_{1}\right)  $ are the alternating Bernoulli
function and Dedekind character sum defined in \cite{9}. Finally put
\[
A_{1}\left(  z,s;\chi,G^{\ast}\right)  =\sum\limits_{m=1}^{\infty}%
\sum\limits_{n=1}^{\infty}\left(  -1\right)  ^{m}\chi\left(  m\right)
G\left(  n+\frac{k}{2},\chi\right)  e^{2\pi inmz/k}n^{s-1}.
\]
Now using (\ref{s1}) and (\ref{s2}) with the notations above we deduce from
(\ref{17}) that
\begin{align}
&  A_{1}\left(  Vz,0;\chi_{1},G_{2}^{\ast}\right)  -\chi_{1}\left(  c\right)
\chi_{2}\left(  b\right)  A_{1}\left(  z,0;\chi_{2},G_{1}^{\ast}\right)
\nonumber\\
&  =\pi i\chi_{1}\left(  -c\right)  \chi_{2}\left(  b\right)  s^{\ast}\left(
d,c;\chi_{2},\chi_{1}\right)  -\frac{\pi i}{c(cz+d)}B_{0}^{\ast}%
(\overline{\chi_{2}})B_{2}^{\ast}(\overline{\chi_{1}})\nonumber\\
&  \quad-\frac{2\pi i}{c}\left(  cz+d\right)  \chi_{2}\left(  -b\right)
B_{0}^{\ast}(\overline{\chi_{1}})P_{2}^{\ast}\left(  0,\overline{\chi_{2}%
}\right)  .\label{s8}%
\end{align}
Notice that when $\chi_{1}=\chi=\overline{\chi_{2}}$\ is primitive, (\ref{s8})
has been given by Meyer \cite[Theorem 3]{9}. Similar to (\ref{s4-a}), one has
\[
A_{1}\left(  z,s;\chi_{1},G_{2}^{\ast}\right)  =\sum\limits_{n=1}^{\infty
}\frac{G\left(  n+\frac{k}{2},\chi_{2}\right)  G\left(  nz+\frac{k}{2}%
,\chi_{1}\right)  }{1-e^{2\pi inz}}n^{s-1}.
\]
Set $z=\pi i/k\gamma$, where $\gamma>0,$ and determine $\theta>0$ by
$\gamma\theta=\pi^{2}/k^{2}.$ Then, for $Vz=-1/z,$ (\ref{s8}) becomes
\begin{align}
&  \sum\limits_{n=1}^{\infty}\frac{G\left(  n+\frac{k}{2},\chi_{2}\right)
G\left(  \frac{in\gamma k}{\pi}+\frac{k}{2},\chi_{1}\right)  }{n\left(
1-e^{-2nk\gamma}\right)  }-\chi_{2}\left(  -1\right)  \sum\limits_{n=1}%
^{\infty}\frac{G\left(  n+\frac{k}{2},\chi_{1}\right)  G\left(  \frac{in\theta
k}{\pi}+\frac{k}{2},\chi_{2}\right)  }{n\left(  1-e^{-2nk\theta}\right)
}\nonumber\\
&  =\pi iP_{1}^{\ast}(0,\overline{\chi_{2}})P_{1}^{\ast}(0,\overline{\chi_{1}%
})-k\gamma B_{0}^{\ast}(\overline{\chi_{2}})B_{2}^{\ast}(\overline{\chi_{1}%
})+\theta k\chi_{2}\left(  -1\right)  B_{0}^{\ast}(\overline{\chi_{1}}%
)B_{2}^{\ast}\left(  \overline{\chi_{2}}\right)  .\label{s8-1}%
\end{align}

\begin{example}
\label{ex1}Let, in (\ref{s8-1}), $\chi_{1}=\chi_{2}=\chi$ be the primitive
character of modulus $k=4$ given by (\ref{pk}). From \cite[Proposition 1]{9}
\[
\sum\limits_{v=0}^{k-1}\left(  -1\right)  ^{v}\chi_{1}(v)=0\text{ for even
}k\text{ and }\chi\not =\chi_{0}%
\]
we have $B_{0}^{\ast}(\overline{\chi})=0.$ Using $G\left(  nz+k/2,\chi\right)
=e^{\pi inz}\left(  e^{\pi inz/2}-e^{-\pi inz/2}\right)  $ and $G\left(
n+k/2,\chi\right)  =-2i\chi\left(  n\right)  $, and simplifying we deduce
that, for $\gamma\theta=\pi^{2}$\textbf{\ }%
\begin{equation}
\sum\limits_{n=1}^{\infty}\frac{\chi\left(  n\right)  }{n\left(  e^{n\gamma
/2}+e^{-n\gamma/2}\right)  }+\sum\limits_{n=1}^{\infty}\frac{\chi\left(
n\right)  }{n\left(  e^{n\theta/2}+e^{-n\theta/2}\right)  }=\frac{\pi}%
{8},\nonumber
\end{equation}
which is special case of (\ref{s4-3}).
\end{example}

\begin{example}
Let $\chi_{1}=\chi_{2}=\chi_{0}$ with modulus $k=8$ in (\ref{s8-1})$.$ Some
arguments give $P_{1}^{\ast}(0,\chi_{0})=0,$ $B_{0}^{\ast}(\chi_{0})=-4,$
$B_{2}^{\ast}(\chi_{0})=-2\left(  B_{2}(1/8)+B_{2}(3/8)\right)  $ and
\[
G\left(  nz+k/2,\chi\right)  =-e^{2\pi inz}\left(  e^{\pi inz/4}+e^{-\pi
inz/4}\right)  \left(  1+e^{\pi inz}\right)  .
\]
Hence, with some simplifications, for $\gamma\theta=\pi^{2}/4,$%
\[
\sum\limits_{n=1}^{\infty}\frac{\left(  -1\right)  ^{n}e^{-12n\gamma}}%
{n\sinh\left(  2n\gamma\right)  }-\sum\limits_{n=1}^{\infty}\frac{\left(
-1\right)  ^{n}e^{-12n\theta}}{n\sinh\left(  2n\theta\right)  }=32\left(
\theta-\gamma\right)  \left(  \frac{11}{192}-\frac{13}{192}\right)  =\frac
{1}{3}\left(  \gamma-\theta\right)  .
\]

\end{example}

\begin{example}
Let, in (\ref{s8-1}), $\chi_{1}=\chi_{2}=\chi$ be the Dirichlet character of
modulus $k=6$ given by (\ref{s9-3}). Using $B_{0}^{\ast}(\overline{\chi})=0$
and $G\left(  nz+k/2,\chi\right)  =e^{\pi inz}\left(  e^{2\pi inz/3}-e^{-2\pi
inz/3}\right)  $, and simplifying we deduce, for $\gamma\theta=\pi^{2}/9,$
that%
\[
\sum\limits_{n=1}^{\infty}\left(  -1\right)  ^{n}\frac{\sin\left(  2\pi
n/3\right)  \cosh\left(  n\gamma\right)  }{n\left(  2\cosh\left(
2n\gamma\right)  +1\right)  }+\sum\limits_{n=1}^{\infty}\left(  -1\right)
^{n}\frac{\sin\left(  2\pi n/3\right)  \cosh\left(  n\theta\right)  }{n\left(
2\cosh\left(  2n\theta\right)  +1\right)  }=\frac{\pi}{9}.
\]

\end{example}

\section{Some examples}

In this final section, we deal with Theorem \ref{rep1} for the special values
of $A=\left\{  f(n)\right\}  $ and $B=\left\{  f^{\ast}(n)\right\}  .$
$\chi_{1}$ and $\chi_{2}$\ will denote Dirichlet characters of modulus $k,$
and $\chi_{0}$\ still stands for principle character.

\textbf{1.} When $A=B=I,$
\[
s\left(  d,c;I,I\right)  =\sum\limits_{n=1}^{c}P_{1}\left(  \frac{n}%
{c}\right)  P_{1}\left(  \frac{dn}{c}\right)  =\sum\limits_{n=1}^{c-1}\left(
\left(  \frac{n}{c}\right)  \right)  \left(  \left(  \frac{dn}{c}\right)
\right)  +\frac{1}{4}=s\left(  d,c\right)  +\frac{1}{4}%
\]
and from (\ref{51}) we have $s\left(  -c,d;I,I\right)  =-s\left(  c,d\right)
+1/4.$ So, Theorem \ref{rep1} reduces to the reciprocity formula for the
classical Dedekind sums given by (\ref{35}).

\textbf{2.} Let $A=\chi_{1}=\left\{  \chi_{1}(n)\right\}  $ and $B=\chi
_{2}=\left\{  \chi_{2}(n)\right\}  .$ Then, using that $s\left(
c,d;A_{c},B_{b}\right)  =\chi_{1}\left(  c\right)  \chi_{2}\left(  b\right)
s\left(  c,d;\chi_{1},\chi_{2}\right)  $\ and $P_{m}\left(  0,A_{c}\right)
=\chi\left(  c\right)  P_{m}\left(  0,\chi_{1}\right)  =\chi\left(  c\right)
B_{m}\left(  \chi_{1}\right)  /m!$, we have
\begin{align}
&  s\left(  c,d;\chi_{1},\chi_{2}\right)  +s\left(  d,c;\chi_{2},\chi
_{1}\right) \nonumber\\
&  =-B_{1}(\chi_{1})B_{1}(\chi_{2})+\left(  \frac{\overline{\chi_{1}}\left(
c\right)  }{c}+c\right)  \frac{\overline{\chi_{2}}\left(  b\right)  }{2d}%
B_{0}\left(  \chi_{2}\right)  B_{2}(\chi_{1})+\frac{d}{2c}\overline{\chi_{1}%
}\left(  c\right)  B_{0}(\chi_{1})B_{2}(\chi_{2}),\label{rp1}%
\end{align}
where we assume that $\chi_{1}\left(  -1\right)  \chi_{2}\left(  -1\right)
=1,$ otherwise $s\left(  c,d;\chi_{1},\chi_{2}\right)  =0$.\ This can be
simplified by using (\ref{58}).

In the case when $\chi_{1}$ an $\chi_{2}$ are primitive, $s\left(
d,c;\chi_{1},\chi_{2}\right)  $ arises in the transformation formula derived
by Berndt \cite{7} and is generalized in \cite{18}. The sum $s\left(
d,c;\chi,\overline{\chi}\right)  $ is first defined in \cite{2} and
\cite[Section 7]{15} according to $\chi$ is primitive or non-principle
Dirichlet character.

\textbf{3.} Let $A=\chi_{1}=\left\{  \chi_{1}(n)\right\}  $ and $B=G_{2}%
=\left\{  G(n,\chi_{2})\right\}  .$ Using that $B_{0}\left(  B\right)
=B_{0}\left(  G_{2}\right)  =0,$ $P_{m}\left(  x,B_{b}\right)  =\overline
{\chi_{2}}\left(  b\right)  P_{m}\left(  x,G_{2}\right)  $ and $s\left(
d,c;B_{b},A_{c}\right)  =\chi_{1}\left(  c\right)  \overline{\chi_{2}}\left(
b\right)  s\left(  d,c;G_{2},\chi_{1}\right)  ,$\textbf{\ }Theorem \ref{rep1}
becomes\textbf{\ }%
\[
s\left(  c,d;\chi_{1},G_{2}\right)  +s\left(  d,c;G_{2},\chi_{1}\right)
=-\chi_{1}\left(  -1\right)  \chi_{2}\left(  -1\right)  P_{1}\left(
0,\chi_{1}\right)  P_{1}\left(  0,G_{2}\right)
\]
when $\chi_{1}$ and $\chi_{2}$ are non-principle, and%
\[
s\left(  c,d;\chi_{0},c_{k}\right)  +s\left(  d,c;c_{k},\chi_{0}\right)
=\frac{d}{c}B_{0}(\chi_{0})P_{2}\left(  0,c_{k}\right)
\]
when $\chi_{1}=\chi_{2}=\chi_{0}$\textbf{.}

Note that for $k=2,$ provided that $d$ is even with $\left(  d,c\right)  =1,$
the sums $s\left(  c,d;\chi_{0},c_{2}\right)  $ and $s\left(  d,c;c_{2}%
,\chi_{0}\right)  $ can be expressed in terms of the sums $s\left(
c,d\right)  $ and $s_{2}\left(  d,c\right)  $ as%
\begin{align}
s\left(  c,d;\chi_{0},c_{2}\right)   &  =2s\left(  c,2d\right)  -3s\left(
c,d\right)  +s\left(  2c,d\right)  ,\nonumber\\
s\left(  d,c;c_{2},\chi_{0}\right)   &  =s_{2}\left(  2d,2c\right)
-s_{2}\left(  d,2c\right)  ,\label{57}%
\end{align}
where $s_{2}\left(  d,c\right)  $ is one of the Hardy--Berndt sums defined by
\cite{22,24}%
\[
s_{2}\left(  d,c\right)  =\sum\limits_{n=1}^{c-1}\left(  -1\right)  ^{n}%
P_{1}\left(  \dfrac{n}{c}\right)  P_{1}\left(  \dfrac{dn}{c}\right)  .
\]
These expressions follow from the fact $P_{1}\left(  x+1/2\right)
=P_{1}\left(  2x\right)  -P_{1}\left(  x\right)  .$

\begin{remark}
If we set $A=G_{2}=\left\{  G(n,\chi_{2})\right\}  $ and $B=\chi_{1}=\left\{
\chi_{1}(n)\right\}  $, then with the use of $B_{0}\left(  A\right)
=B_{0}\left(  G_{2}\right)  =0,$ $P_{m}\left(  x,A_{c}\right)  =\overline
{\chi_{2}}\left(  c\right)  P_{m}\left(  x,G_{2}\right)  $ and $s\left(
d,c;B_{b},A_{c}\right)  =\chi_{1}\left(  b\right)  \overline{\chi_{2}}\left(
c\right)  s\left(  d,c;\chi_{1},G_{2}\right)  ,$ Theorem \ref{rep1} would
become\textbf{\ }%
\[
s\left(  c,d;G_{2},\chi_{1}\right)  +s\left(  d,c;\chi_{1},G_{2}\right)
=-\chi_{1}\left(  -1\right)  \chi_{2}\left(  -1\right)  P_{1}\left(
0,\chi_{1}\right)  P_{1}\left(  0,G_{2}\right)
\]
when $\chi_{1}$ and $\chi_{2}$ are non-principle, and
\[
s\left(  c,d;c_{k},\chi_{0}\right)  +s\left(  d,c;\chi_{0},c_{k}\right)
=\left(  \frac{1}{c}+c\right)  \frac{1}{d}B_{0}(\chi_{0})P_{2}\left(
0,c_{k}\right)
\]
when $\chi_{1}=\chi_{2}=\chi_{0}$\textbf{.}
\end{remark}

\textbf{4. }Let $A=\left\{  G(n,\chi_{1})\right\}  =G_{1}$ and $B=\left\{
G(n,\chi_{2})\right\}  =G_{2}.$ One has $A_{c}=\left\{  G(cn,\chi
_{1})\right\}  =\left\{  \overline{\chi_{1}}\left(  c\right)  G(n,\chi
_{1})\right\}  =\overline{\chi_{1}}\left(  c\right)  G_{1}$ and $B_{c}%
=\overline{\chi_{2}}\left(  c\right)  G_{2}$ for $\left(  c,k\right)  =1.$
Thus, we have $B_{0}\left(  A\right)  =B_{0}\left(  B\right)  =0,$
$P_{1}\left(  0,A_{-c}\right)  =\overline{\chi_{1}}\left(  -c\right)
P_{1}\left(  0,G_{1}\right)  $ and $s\left(  d,c;B_{b},A_{c}\right)
=\overline{\chi_{1}}\left(  c\right)  \overline{\chi_{2}}\left(  b\right)
s\left(  d,c;G_{2},G_{1}\right)  $. Then, for $\chi_{1}=\chi_{2}=\chi_{0},$
\[
s\left(  c,d;c_{k},c_{k}\right)  +s\left(  d,c;c_{k},c_{k}\right)  =\frac
{1}{4}\phi\left(  k\right)  \phi\left(  k\right)
\]
and for $\chi_{1}\not =\chi_{0}$ and $\chi_{2}\not =\chi_{0},$
\[
s\left(  c,d;G_{1},G_{2}\right)  +s\left(  d,c;G_{2},G_{1}\right)
=-P_{1}\left(  0,G_{1}\right)  P_{1}\left(  0,G_{2}\right)  .
\]

In particular, for $k=2$, provided that $d$ is even with $\left(  d,c\right)
=1,$%
\[
s\left(  d,c;c_{2},c_{2}\right)  =2s_{2}\left(  d,2c\right)  -s_{2}\left(
2d,2c\right)  +\frac{1}{4}.
\]

\textbf{5. }$A=\left\{  \left(  -1\right)  ^{n}\chi_{1}(n)\right\}  $ and
$B=\left\{  \left(  -1\right)  ^{n}\chi_{2}(n)\right\}  .$ If $k$ is even,
then the sequences $A$ and $B$ have period $h=k.$ If $k$ is odd, then $A$ and
$B$ have period $h=2k.$ In both cases $b$ and $c$ must be odd since $ad-bc=1$
and $a\equiv d\equiv0\left(  \operatorname{mod}h\right)  .$ Using (\ref{s8-3})
and (\ref{s8-4}) in Theorem \ref{rep1} gives
\begin{align*}
&  s^{\ast}\left(  c,d;\chi_{1},\chi_{2}\right)  +\chi_{1}\left(  -1\right)
\chi_{2}\left(  -1\right)  s^{\ast}\left(  d,c;\chi_{2},\chi_{1}\right) \\
&  =-\chi_{2}\left(  -1\right)  P_{1}^{\ast}\left(  0,\overline{\chi_{2}%
}\right)  P_{1}^{\ast}\left(  0,\overline{\chi_{1}}\right)  +\frac{1}%
{2dc}\overline{\chi_{1}}\left(  -c\right)  \overline{\chi_{2}}\left(
-b\right)  B_{0}^{\ast}(\overline{\chi_{2}})B_{2}^{\ast}(\overline{\chi_{1}%
})\\
&  \quad+\frac{d}{c}\overline{\chi_{1}}\left(  c\right)  B_{0}^{\ast
}(\overline{\chi_{1}})P_{2}^{\ast}\left(  0,\overline{\chi_{2}}\right)
+\frac{c}{d}\chi_{1}\left(  -1\right)  \overline{\chi_{2}}\left(  -b\right)
B_{0}^{\ast}(\overline{\chi_{2}})P_{2}^{\ast}\left(  0,\overline{\chi_{1}%
}\right)  .
\end{align*}
For non-principle characters $\chi_{1}$ and $\chi_{2},$ this reduces to
\begin{equation}
s^{\ast}\left(  c,d;\chi_{1},\chi_{2}\right)  +\chi_{1}\left(  -1\right)
\chi_{2}\left(  -1\right)  s^{\ast}\left(  d,c;\chi_{2},\chi_{1}\right)
=-\chi_{2}\left(  -1\right)  P_{1}^{\ast}\left(  0,\overline{\chi_{2}}\right)
P_{1}^{\ast}\left(  0,\overline{\chi_{1}}\right) \label{59}%
\end{equation}
since $B_{0}^{\ast}(\overline{\chi_{1}})=B_{0}^{\ast}(\overline{\chi_{2}})=0.
$ Remark that (\ref{59}) has been obtained by Meyer \cite[Theorem 4]{9} when
$\overline{\chi_{1}}=\chi_{2}=\chi$ is primitive character of modulus $h=k$
($k$ even).

\textbf{6.} Setting $A=\left\{  e^{2\pi in/k}\right\}  $ and $B=\widehat{A}%
=\left\{  \widehat{f}(n)\right\}  .$ In view of (\ref{47}), we have
\[
\widehat{f}(n)=\frac{1}{k}\sum\limits_{j=0}^{k-1}e^{2\pi i\left(  1-n\right)
j/k}=%
\begin{cases}
1, & \text{if }n\equiv1\left(  \operatorname{mod}k\right) \\
0, & \text{if }n\not \equiv 1\left(  \operatorname{mod}k\right)  \text{.}%
\end{cases}
\]
Then using that $B_{0}\left(  A\right)  =1,$ $B_{0}\left(  \widehat{A}\right)
=0,$ $P_{1}\left(  0,\widehat{A}_{-b}\right)  =-P_{1}\left(  c/k\right)  ,$
$P_{2}\left(  0,\widehat{A}_{b}\right)  =P_{2}\left(  c/k\right)  ,$
$P_{1}\left(  0,A_{-c}\right)  =-1/2-\left(  i/2\right)  \cot\left(  \pi
c/k\right)  $\ and%
\[
P_{1}\left(  \dfrac{cn}{d},\widehat{A}_{b}\right)  =\sum\limits_{v=0}%
^{k-1}\widehat{f}(-bv)P_{1}\left(  \dfrac{v+cn/d}{k}\right)  =P_{1}\left(
\dfrac{c+cn/d}{k}\right)  ,
\]
reciprocity formula becomes%
\[
s\left(  c,d;A_{c},\widehat{A}_{b}\right)  +s\left(  d,c;\widehat{A}_{b}%
,A_{c}\right)  =-P_{1}\left(  c/k\right)  \left(  \frac{1}{2}+\frac{i}{2}%
\cot\frac{\pi c}{k}\right)  +\frac{d}{c}P_{2}\left(  \dfrac{c}{k}\right)  .
\]
In this case the periodic Dedekind sums take form
\begin{align*}
s\left(  c,d;A_{c},\widehat{A}_{b}\right)   &  =\sum\limits_{n=1}^{dk}e^{2\pi
icn/k}P_{1}\left(  \dfrac{n}{dk}\right)  P_{1}\left(  \dfrac{c\left(
d+n\right)  }{dk}\right)  ,\\
s\left(  d,c;\widehat{A}_{b},A_{c}\right)   &  =\sum\limits_{\mu=0}^{c-1}%
P_{1}\left(  \dfrac{\mu}{c}-\frac{1}{k}\right)  P_{1}\left(  \dfrac{dk\mu}%
{c},A_{c}\right)
\end{align*}
since $d\equiv0\left(  \operatorname{mod}k\right)  $.

If we take $k=2,$ provided that $c$ is odd, then we have\textbf{\ }%
\begin{equation}
s\left(  c,d;A_{c};\widehat{A}_{b}\right)  +s\left(  d,c;\widehat{A}_{b}%
;A_{c}\right)  =-\frac{d}{24c},\label{52}%
\end{equation}
where we have used that $P_{1}\left(  c/2\right)  =P_{1}\left(  1/2\right)
=0$ and $2P_{2}\left(  c/2\right)  =2P_{2}\left(  1/2\right)  =B_{2}\left(
1/2\right)  =\left(  1/2\right)  ^{2}-1/2+1/6=-1/12.$

\begin{remark}
Using $P_{1}\left(  x+\frac{1}{2}\right)  =P_{1}\left(  2x\right)
-P_{1}\left(  x\right)  $, we have
\[
s\left(  c,d;A_{c};\widehat{A}_{b}\right)  =s_{2}\left(  2c,2d\right)
-s_{2}\left(  c,2d\right)
\]
and from \cite[Eq.(3.4)]{23}%
\begin{align*}
s\left(  d,c;\widehat{A}_{b};A_{c}\right)   &  =2\sum\limits_{\mu=1}%
^{c-1}P_{1}\left(  \dfrac{\mu}{c}+\frac{1}{2}\right)  P_{1}\left(  \dfrac
{d\mu}{c}\right)  -\sum\limits_{\mu=1}^{c-1}P_{1}\left(  \dfrac{\mu}{c}%
+\frac{1}{2}\right)  P_{1}\left(  \dfrac{2d\mu}{c}\right) \\
&  =2s\left(  d,c\right)  +s_{3}\left(  \frac{d}{2},c\right)  -s\left(
2d,c\right)  -\frac{1}{2}s_{3}\left(  d,c\right)  ,
\end{align*}
where $s_{3}\left(  d,c\right)  $ is one of the Hardy--Berndt sums defined by
\cite{22,24}%
\[
s_{3}\left(  d,c\right)  =\sum\limits_{n=1}^{c-1}\left(  -1\right)  ^{n}%
P_{1}\left(  \frac{dn}{c}\right)  .
\]
Then (\ref{52}) becomes%
\[
2s\left(  d,c\right)  -s\left(  2d,c\right)  +s_{2}\left(  2c,2d\right)
-s_{2}\left(  c,2d\right)  +s_{3}\left(  d/2,c\right)  -\frac{1}{2}%
s_{3}\left(  d,c\right)  =-\frac{d}{24c}.
\]

\end{remark}

\end{document}